\documentclass[final,onefignum,onetabnum]{siamonline171218}
\newsavebox{\mybox} 
\usepackage{tikz}
\newcommand*{\circled}[1]{\lower.7ex\hbox{\tikz\draw (0pt, 0pt)%
		circle (.5em) node {\makebox[1em][c]{\small #1}};}}
\newtheorem{example}{Example}[section]

\usepackage{algorithmic}
\usepackage[algo2e,linesnumbered,lined,ruled]{algorithm2e} %
\usepackage{enumitem} %
\setlist[enumerate]{leftmargin=.5in}
\setlist[itemize]{leftmargin=.5in}
\usepackage{setspace}
\newcommand{\reffig}[1]{Fig.\,\ref{#1}}
\newcommand{\reffigs}[1]{Figs.\,\ref{#1}}
\newcommand{\reftab}[1]{Tab.\,\ref{#1}}
\newcommand{\reftabs}[1]{Tabs.\,\ref{#1}}
\newcommand{\refalg}[1]{Algorithm\,\ref{#1}}
\newcommand{\refalgs}[1]{Algorithms\,\ref{#1}}
\newcommand{\refrmk}[1]{Remark\,\ref{#1}}

\usepackage{cite} % 
\usepackage{mathrsfs}
\usepackage{amsmath,amssymb,amsfonts} % 
\usepackage{cases} % 
\usepackage{graphicx,subfigure}
\usepackage{epstopdf} % 
\usepackage{bm}
\usepackage{setspace}
\usepackage{multirow}
\usepackage{booktabs} %\toprule, \midrule and \bottomrule
\usepackage{extarrows}
\usepackage{relsize}
\usepackage[section]{placeins} 
\usepackage{xcolor,float} 
\makeatletter
\newcommand{\rmnum}[1]{\romannumeral #1}
\newcommand{\Rmnum}[1]{\expandafter\@slowromancap\romannumeral #1@}
\newtheorem{remark}{Remark}

\makeatother
\graphicspath{{fig/}}

\makeatletter
\newcommand*\bigcdot{\mathpalette\bigcdot@{.5}}
\newcommand*\bigcdot@[2]{\mathbin{\vcenter{\hbox{\scalebox{#2}{$\m@th#1\bullet$}}}}}
\makeatother

\title{Parallel Multi-Stage Preconditioners with Adaptive Setup for the Black Oil Model\thanks{This is a preprint manuscript, which is submitted to Computers and Geosciences Journal.}
}

\author{Li Zhao\thanks{School of Mathematics and Computational Science, Xiangtan University, Xiangtan, Hunan 411105, P. R. China} 
	\and Chunsheng Feng\thanks{National Center for Applied Mathematics in Hunan, Xiangtan University, Xiangtan 411105, P. R. China; Hunan Shaofeng Institute for Applied Mathematics, Xiangtan 411105, P. R. China.}
	\and Chensong Zhang\thanks{Academy of Mathematics and System Sciences, Beijing 100190, P. R. China. (Corresponding author: Chensong Zhang, Email: \email{zhangcs@lsec.cc.ac.cn}).}
	\and Shi Shu\footnotemark[2]
}

\date{\currenttime}
\begin{document}		
\maketitle 
%\tableofcontents

%\runninglinenumbers\linenumbers
\newcommand{\slugmaster}{%
\slugger{sifin}{}{}{}{}}%{xxxx}{xx}{x}{x--x}}%slugger should be set to juq, siads, sifin, or siims
%\listoffigures
%\captionsetup[figure]{labelformat={default},labelsep=period,name={Fig.}}
%\captionsetup[table]{name={Tab.}}
\def\figurename{{Fig.}}
\def\tablename{{Tab.}}

\begin{abstract}
The black oil model is widely used to describe multiphase porous media flow in the petroleum industry. The fully implicit method features strong stability and weak constraints on timestep sizes; hence, it is commonly used in current mainstream commercial reservoir simulators. In this paper, a Constrained Pressure Residual (CPR) preconditioner with an adaptive ``setup phase" is developed to improve the parallel efficiency of a petroleum reservoir simulation. Furthermore, we propose a multi-color Gauss--Seidel (GS) algorithm for the algebraic multigrid method based on the coefficient matrix of strong connections. Numerical experiments show that the proposed preconditioner can improve the parallel performance for both OpenMP and Compute Unified Device Architecture (CUDA) implements. Moreover, the proposed algorithm yields good parallel speedup as well as the same convergence behavior as the corresponding single-thread algorithm. In particular, for a three-phase benchmark problem (about 3.28 million degrees of freedom), the parallel speedup of the OpenMP version is over 6.5 with 16 threads, and the CUDA version reaches more than 9.5.
\end{abstract}
% REQUIRED

% REQUIRED

\pagestyle{myheadings}
\thispagestyle{plain}
\markboth{Parallel Multi-Stage Preconditioners with Adaptive Setup for the Black Oil Model}
{Li Zhao, Chunsheng Feng, Chensong Zhang and Shi Shu}

\section{Introduction}\label{sec:1}
Research on petroleum reservoir simulation can be traced back to the 1950s. To describe and predict the transportation of hydrocarbons, various mathematical models have been established, such as the black oil model, compositional model, thermal recovery model, and chemical flooding model~\cite{Aziz1979PetroleumRS,ChenHuanMa,GOUDARZI201696,Peaceman1977,VALIOLLAHI2012}. The black oil model consists of multiple coupled nonlinear partial differential equations (PDEs). It is a fundamental mathematical model to describe the three-phase flow in petroleum reservoirs and is widely used in simulating primary and secondary recovery.

After 70 years of development, there is a large body of research on the numerical methods of the black oil model, including the Simultaneous Solution (SS) method \cite{1959One}, Fully Implicit Method (FIM) \cite{DouglasFIMI}, IMplicit Pressure Explicit Saturation (IMPES) method \cite{StoneIMPES}, and Adaptive Implicit Method (AIM) \cite {AIM}.
Compared with other methods, FIM is commonly used in mainstream commercial reservoir simulators because of its unconditional stability with respect to timestep sizes. However, a coupled Jacobian linear algebraic system needs to be solved in each Newton iteration step. Owing to the complexity of the practical engineering problem, such systems are difficult to solve with traditional linear solvers. In the reservoir simulation, the solution time of Jacobian systems easily occupies more than 80\% of the whole simulation time. Therefore, how to efficiently solve coupled Jacobian systems, especially on modern computers, is a problem that still attracts a lot attention today.

Typically, linear solution methods can be divided into two phases, the ``setup phase" (SETUP) and the ``solve phase" (SOLVE). These methods can usually be categorized as direct methods \cite{2006Direct} and iterative methods \cite{2003Iterative}. Compared with direct methods, iterative methods have the advantages of low memory and computation complexity and potentially good parallel scalability \cite{osti10249810}. The linear algebraic systems arising from fully implicit petroleum reservoir simulation are usually solved by iterative methods. In particular, Krylov subspace methods \cite{2003Iterative} (e.g., GMRES and BiCGstab) are frequently adopted. For ill-conditioned linear systems, the preconditioning technique \cite{Xu1992MSC} is needed to accelerate the convergence of the iterative methods. The preconditioners for reservoir simulation include: Incomplete LU (ILU) factorization~\cite{10.2118/12262-MS}, Algebraic MultiGrid (AMG)~\cite{1984Algebraic,falgout2006an}, Constrained Pressure Residual (CPR)~\cite{10.2118/96809-MS,2017Numerical,IGE1983,10.2118/13536-MS}, and Multi-Stage Preconditioner (MSP)~\cite{10.2118/118722-MS, XiaoZheHU2013,10.2118/105832-MS}. 
The ILU method is relatively easy to implement, but as the problem size increases, its convergence deteriorates. The advantages of the AMG method are that it is easy to use and effective on elliptic problems. Owing to the asymmetry, heterogeneity, and nonlinear coupling feature of petroleum reservoir problems, the performance of the AMG method also deteriorates. The CPR method combines the advantages of ILU and AMG, and the MSP method is a generalization of CPR. %Thus, the CPR and MSP preconditioners are widely used in petroleum reservoir problems.

As multi-core and many-core architectures have become more popular, parallel computing for petroleum reservoir simulation is now a subject of great interest. In recent years, there has been some work on parallel algorithms for reservoir problems~\cite{BUCKER2008,Dogru2009,2014A2,MESBAH2019574,2010High,WEI2015,WILKINS2021,2014A1,ShuhongWu2016,YangBo2016,YANG20192}, and the references therein. For example, \cite{2014A2} designed an OpenMP parallel algorithm with high efficiency and a low memory cost for standard interpolation and coarse grid operator of AMG, under the framework of Fast Auxiliary Space Preconditioning (FASP, \url{http://www.multigrid.org/fasp/}). \cite{2014A1,ShuhongWu2016} developed a Method of Subspace Correction (MSC) based on \cite{2014A2} and realized a cost-effective OpenMP parallel reservoir numerical simulation. \cite{2010High} designed a GPU parallel algorithm based on the METIS  partition for the IMPES method. \cite{YangBo2016} studied the GPU parallel algorithm of ILU and AMG based on a hybrid sparse storage format.

In this paper, we focus on the solution method for the linear algebraic systems arising from the fully implicit discretization of the black oil model, aiming to improve the parallel efficiency of the CPR preconditioner. The main contributions of this work are listed as follows:

\begin{itemize}
	\item We propose an adaptive SETUP CPR preconditioner (denoted as ASCPR) to improve the efficiency and parallel performance of the solver. 
	A practical adaptive criterion is proposed to judge whether a new SETUP is necessary. The technology can bring two benefits: (1) The efficiency of the solver is improved because the number of SETUP calls can be significantly reduced; (2) the parallel performance is improved because there are many essentially sequential algorithms in the SETUP (i.e., the parallel speedup of these algorithms is low).
	
	\item We propose an efficient parallel algorithm for the Gauss-Seidel (GS) relaxation in AMG methods. Starting from the strong connections of coefficient matrix, we design an algorithm for algebraic multi-color grouping. The algorithm has two desirable features: (1) Not relying on the grid (completely transformed into algebraic behavior); (2) yielding the same convergence behavior as the corresponding single-thread algorithm. 
	Furthermore, we use an adjustable strength threshold to filter small matrix entries (enhancing the sparseness) to improve the parallel performance of the algorithm.
\end{itemize}

The rest of the paper is organized as follows. Section \ref{sec:2} introduces the black oil model and its fully implicit discrete systems. Section \ref{sec:3} reviews the CPR-type preconditioners. In Section \ref{sec:4}, an adaptive SETUP CPR preconditioner is proposed. In Section \ref{sec:5}, the parallel implementation of multi-color GS based on the coefficient matrix of strong connections is given. In Section \ref{sec:6}, numerical experiments are given. Section \ref{sec:7} provides the summary of the work of this paper.

\section{Preliminaries}\label{sec:2}
\subsection{The black oil model}\label{sec:2-1}
This paper considers the following three-phase standard black oil model of water, oil, or gas in porous media \cite{Aziz1979PetroleumRS,ChenHuanMa,Peaceman1977}. The mass conservation equations of water, oil, and gas, respectively, are

\begin{equation}\label{eq:model-1}
	\frac{\partial}{\partial t}\left(\phi \frac{S_{w}}{B_{w}}\right) = -\nabla \cdot\left(\frac{1}{B_{w}} \bm{u}_{w}\right) + \frac{Q_{W}}{B_{w}},	
\end{equation}
\begin{equation}\label{eq:model-2}
	\frac{\partial}{\partial t}\left(\phi \frac{S_{o}}{B_{o}}\right) = -\nabla \cdot\left(\frac{1}{B_{o}} \bm{u}_{o}\right) + \frac{Q_{O}}{B_{o}},	
\end{equation}
\begin{equation}\label{eq:model-3}
	\frac{\partial}{\partial t}\left[\phi\left(\frac{S_{g}}{B_{g}}+\frac{R_{so} S_{o}}{B_{o}}\right)\right] = -\nabla \cdot\left(\frac{1}{B_{g}} \bm{u}_{g}+\frac{R_{s o}}{B_{o}} \bm{u}_{o}\right) + \frac{Q_{G}}{B_{g}}+\frac{R_{so} Q_{O}}{B_{o}}.
\end{equation}

Here, $ S_{\alpha}$ is the saturation of phase $ \alpha $ ($\alpha=w,o,g $ represents the water phase, oil phase, and gas phase, respectively), $ B_{\alpha} $ is the volume coefficient of phase $ \alpha $, $ \bm{u}_{\alpha} $ is the velocity of phase $ \alpha $, $ \phi $ is the porosity of the rock, $ R_{so} $ is the dissolved gas-oil ratio, and $ Q_{\beta}$ is the injection and production rate of component $\beta $ ($\beta=W,O,G$ represents the water component, oil component, and gas component, respectively) under the ground standard status.

Assume that the three-phase fluid flow in porous media satisfies Darcy's law:
\begin{equation}\label{eq:model-4}
	\bm{u}_{\alpha}=-\frac{\kappa \kappa_{r \alpha}}{\mu_{\alpha}}\left(\nabla P_{\alpha}-\rho_{\alpha} \mathfrak{g} \nabla z\right), \quad \alpha=w, o, g,
\end{equation}
where $ \kappa $ is the absolute permeability, $ \kappa_{r \alpha} $ is the relative permeability of phase $\alpha $, $ \mu_{\alpha} $ is the viscosity coefficient of phase $ \alpha $, $ P_{\alpha }$ is the pressure of phase $ \alpha $, $ \rho_{\alpha} $ is the density of phase $ \alpha $, $ \mathfrak{g} $ is the gravity acceleration, and $ z $ is the depth.% calculated 

The unknown quantities $S_{\alpha}$ and $P_{\alpha}$ in Eqs. \eqref{eq:model-1}--\eqref{eq:model-4} also satisfy the following constitutive relation:
\begin{itemize}
	\item Saturation constraint equation:
	\begin{equation}\label{eq:model-5}
		S_{w}+S_{o}+S_{g}=1.
	\end{equation}
	
	\item Capillary pressure equations:
	\begin{equation}\label{eq:model-6}
		\begin{split}
			P_w &= P_o - P_{cow}, \\
			P_g &= P_o - P_{cgo},
		\end{split}
	\end{equation}
	where $P_{cow} $ is the capillary pressure
	between the oil and water phases, and $P_{cgo} $ is the capillary pressure between the gas and oil phases.
\end{itemize}

\subsection{Discretization and algorithm flowchart}\label{sec:2-2}
The FIM scheme is currently commonly used in mainstream commercial reservoir simulators. This is because the scheme has the characteristics of strong stability and weak constraint on the timestep sizes. These characteristics highlight the advantages of the FIM, especially when the nonlinearity of the models is relatively strong.

In this paper, we use FIM to discretize the governing Eqs. \eqref{eq:model-1}--\eqref{eq:model-3}. That is, the time direction is discretized by the backward Euler method, and the spatial direction is discretized by the upstream weighted central finite difference method~\cite{ChenHuanMa,DouglasFIMI}.
After discretization, the coupled nonlinear algebraic equations are obtained. 
Such equations are linearized by adopting the Newton method to form the Jacobian system $ Ax = b $ of the reservoir equation with implicit wells, namely:
\begin{equation}\label{eq:Jacobian-system}
	\left( {\begin{array}{*{20}{c}}
			{{{{A}}_{RR}}}&{{A_{RW}}}\\
			{{A_{WR}}}&{{A_{WW}}}
	\end{array}} \right)
	\left( {\begin{array}{*{20}{c}}
			{{x_R}}\\
			{{x_W}}
	\end{array}} \right) 
	= \left( {\begin{array}{*{20}{c}}
			{{b_R}}\\
			{{b_W}}
	\end{array}} \right),
\end{equation}
where $ A_{RR} $ and $ A_{RW} $ are the derivatives of the reservoir equations for reservoir variables and well variables, respectively; $ A_{WR} $ and $ A_{WW} $ are the derivatives of the well equations for reservoir variables and well variables, respectively; $ x_R $ and $ x_W $ are reservoir and bottom-hole flowing pressure variables, respectively; and $ b_R $ and $ b_W $ are the right-hand side vectors that correspond to the reservoir fields and the implicit wells, respectively.

The subsystem corresponding to the reservoir equations in the discrete system \eqref{eq:Jacobian-system} is $ A_{RR}x_R=b_R $; that is,
\begin{equation}\label{eq:Jacobian-system-reservoir}
	\left( {\begin{array}{*{20}{c}}
			A_{PP}& A_{PS_w} & A_{PS_o}\\
			A_{S_wP}& A_{S_wS_w} & A_{S_wS_o}\\
			A_{S_oP}& A_{S_oS_w} & A_{S_oS_o}\\
	\end{array}} \right)%,
	\left( {\begin{array}{*{20}{c}}
			x_{P}     \\
			x_{S_w} \\
			x_{S_o} \\
	\end{array}} \right)%, 
	= 
	\left( {\begin{array}{*{20}{c}}
			b_{P}     \\
			b_{S_w} \\
			b_{S_o}
	\end{array}} \right),
\end{equation}
where $ P, S_w$, and $S_o $ are primary variables corresponding to oil pressure, water saturation, and oil saturation, respectively.

\begin{remark}
	\rm{For convenience, we do not describe how to deal with well equations.}
\end{remark}

In the following, we present a general algorithm flowchart of the petroleum reservoir simulation; see \reffig{fig:total-flow}. 
\begin{figure}[H]
	\centering
	\includegraphics[width=0.45\linewidth]{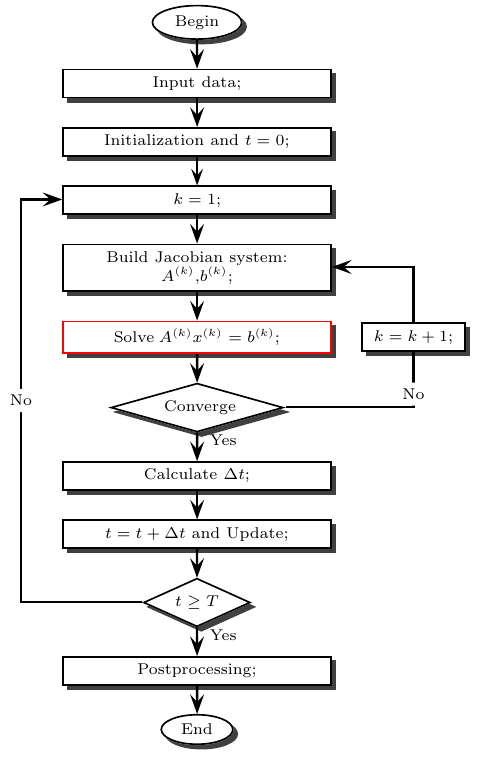}
	\caption{Algorithm flowchart of the petroleum reservoir simulation.}
	\label{fig:total-flow}
\end{figure}

According to \reffig{fig:total-flow}, the algorithm flowchart includes two loops: the outer loop (time marching) and the inner loop (Newton iterations). 
In each Newton iteration, a Jacobian system $A^{(k)} x^{(k)} = b^{(k)}$ (superscript $k$ is the number of Newton iterations) needs to be solved, which is the main computational work to be carried out.

\section{The CPR-type preconditioners}\label{sec:3}
The primary variables usually consists of oil pressure $ P $ and saturations $ S $ (including $ S_w $ and $ S_o $) in FIM, which have different mathematical properties, respectively. For example, the pressure equation is parabolic, and the saturation equation is hyperbolic \cite{Trangenstein1989}. These properties provide a theoretical basis for the design of multiplicative subspace correction methods \cite{Xu1992MSC}.

\subsection{CPR preconditioner}\label{sec:3-1}
First, the transfer operator $ \varPi_P: \mathcal{V}_P \rightarrow \mathcal{V} $ is defined, where $ \mathcal{V}_P $ and $\mathcal{V}$  are the pressure variables space and the variables space of the whole reservoir, respectively. Next, a well-known two-stage preconditioner, the Constrained Pressure Residual (CPR)~\cite{10.2118/96809-MS,2017Numerical,IGE1983,10.2118/13536-MS} preconditioner $B$, is defined as
\begin{equation}\label{eq:CPR}
	I-B A = (I-RA)(I- \varPi_P B_P \varPi_P^T A),
\end{equation}
where $B_P$ is solved by the AMG method, and the relaxation (or smoothing) operator $ R $ uses the Block ILU (BILU) method \cite{10.2118/12262-MS}.

Finally, the CPR preconditioning algorithm is shown in \refalg{alg:CPR}.

\begin{algorithm}[h]
	\caption{CPR method}\label{alg:CPR}
	\setcounter{AlgoLine}{0}
	\LinesNumbered
	\KwIn{$A, b, x$;}
	\KwOut{$x$;}
	$ r \leftarrow b-Ax $; 
	
	$ x \leftarrow x + \varPi_P B_P \varPi_P^{T}r $;
	
	$ r \leftarrow b-Ax $;
	
	$ x \leftarrow x + Rr $.
	
	\Return $ x $.
\end{algorithm}

\subsection{Red-Black GS method}\label{sec:3-2}
As know, compared with the Jacobi algorithm, the GS algorithm  uses the updated values in the iterative process. Hence, the GS algorithm obtains a better convergence rate and is widely used as a smoother of AMG. Now, the parallel red-black GS (also referred to multi-color GS) algorithm on the structured grid is fairly mature \cite{2003Iterative,feng2014numerical}. We take a 2D structured grid as an example to present two-color and four-color vertex-grouping diagrams; see \reffig{fig:sg-mc}.

\begin{figure}[H]
	\centering
	\subfigure[two-color]{\label{fig:sg-mc-2}
		\begin{minipage}[t]{0.45\linewidth}
			\centering
			\includegraphics[width=4cm]{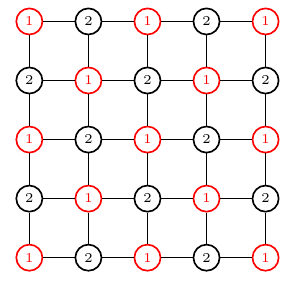}		
		\end{minipage}
	}%	
	\subfigure[four-color]{\label{fig:sg-mc-4}
		\begin{minipage}[t]{0.45\linewidth}
			\centering
			\includegraphics[width=4cm]{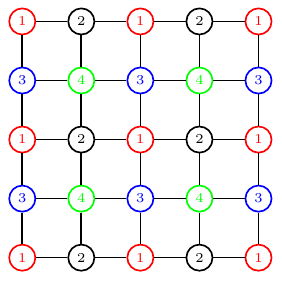}	
		\end{minipage}
	}
	\caption{Two-color and four-color vertex-grouping diagrams.}\label{fig:sg-mc}
\end{figure}

In \reffig{fig:sg-mc-2}, the vertices are divided into two groups and marked as red and black points; that is, vertex set $V$ is divided into $V_{1}$ and $V_{2}$. In \reffig{fig:sg-mc-4}, the vertices are divided into four groups, marked as red, black, blue, and green; that is, the vertex set $V$ is divided into $V_{1}$, $V_{2}$, $V_{3}$, and $V_{4}$. The multi-color GS algorithm aims to perform parallel smoothing on the vertices of the same color; that is, (1) for the case of two colors, first, all-red vertices ($V_{1}$) are smoothed in parallel, and then all-black ($V_{2}$) vertices are smoothed in parallel; (2) for the case of four colors, first, all-red vertices ($V_{1}$) are smoothed in parallel; second, all-black vertices ($V_{2}$) are smoothed in parallel; third, all-blue vertices ($V_{3}$) are smoothed in parallel; and finally, all-green vertices ($V_{4}$) are smoothed in parallel.

Note that different vertices sets are sequential, and the interior of the vertices set is entirely parallel. From the perspective of parallel effects, the above two smooth orderings can yield the same number of iterations as the sequential algorithm. From the scope of application, the two-color grouping is only applicable to the five-point stencil and the four-color grouping can be applied to the nine-point stencil. Similarly, the multi-color GS algorithm of the 2D structured grid can be extended to the 3D structured grid.

The popular parallel variant of GS is the red-black GS algorithm based on structured grids, and it is not suitable for unstructured grids. As a consequence, the application range of the algorithm is limited. Moreover, there is also a hybrid method (combining Jacobi and GS), but its convergence rate deteriorates. From the perspective of parallel implementation, the GS algorithm, an essentially sequential algorithm, is not conducive to yielding the same convergence behavior as the corresponding single-thread algorithm and obtaining high parallel efficiency at the same time.

Finally, we discuss some shortcomings of the standard CPR method. 
\begin{enumerate}
	\item [(\rmnum{1})] Petroleum reservoir simulation is a time-dependent and nonlinear problem. The Jacobian systems need to be solved in each Newton iteration step. The matrix structure of these systems is similar. The CPR method does not take full advantage of the similarity.
	
	\item [(\rmnum{2})] The CPR method contains many essentially sequential steps in the SETUP, which result in low parallel efficiency.
	
	\item [(\rmnum{3})] The GS method is commonly used as the smoother in AMG methods. When we try to improve the parallel performance of the smoother, the convergence rate of AMG methods usually deteriorates.
\end{enumerate}

In view of the shortcomings (\rmnum{1}) and (\rmnum{2}) mentioned above, we discuss how to reuse similar matrix structures to improve the performance of CPR in Section \ref{sec:4}. Furthermore, for the shortcoming (\rmnum{3}), we propose a multi-color GS method from the algebraic point of view in Section \ref{sec:5}. %We show that the multi-color GS method can not only yield the same convergence behavior as the corresponding single-thread algorithm, but can also obtain good parallel efficiency.

\section{An adaptive SETUP CPR preconditioner}\label{sec:4}
In this section, we propose an efficient CPR preconditioner using an adaptive SETUP strategy. For the sake of simplicity, we employ CPR as the preconditioner and restart GMRES as the iterative method (denoted as CPR-GMRES) to illustrate the fundamental idea of the ASCPR preconditioner; the corresponding algorithm is denoted as ASCPR-GMRES.
We develop ASCPR-GMRES to efficiently solve the Jacobian systems and take the right preconditioner as an example to describe its implementation; see \refalg{alg:share-pgmres} and \refalg{alg:share-setup}.

Note that the CPR preconditioner $B^{(k)}$ is generated by exploiting an adaptive strategy in \refalg{alg:share-pgmres}. The concrete implementation of the strategy is as follows; see \refalg{alg:share-setup}. 
\begin{itemize}
	\item If $ k=1 $, the preconditioner $ B^{(1)} $ is generated by calling \refalg{alg:CPR} (\refrmk{rmk:CPR});
	
	\item If $ k>1 $, the establishment of the preconditioner $ B^{(k)}$ can be viewed as the following two steps. 
	First, we obtain $ It^{(k-1)} $, which is the number of iterations obtained by solving the previous Jacobian system $ A^{(k-1)} x^{(k-1)} = b^{(k-1 )}$. Furthermore, there are two situations when judging the size of $ It^{(k-1)} $ and $\mu$ (given a threshold greater than or equal to 0). If $ It^{(k-1)} \leq \mu$, the preconditioner $ B^{(k)} $ adopts the previous preconditioner $ B^{(k-1)} $; otherwise, the preconditioner $ B^{(k)} $ is generated by calling \refalg{alg:CPR}.
\end{itemize}

\begin{remark}\label{rmk:CPR}
	\rm{\refalg{alg:CPR} is a preconditioning method ($w = \text{CPR}(A,g,w_0)$ i.e., $ w=Bg $). For the convenience of describing \refalgs{alg:share-pgmres} and \ref{alg:share-setup}, we assume \refalg{alg:CPR} creates a CPR preconditioner B.}
\end{remark}

\begin{remark}\label{rmk:CPR2}
	\rm{If the sizes of matrices $ A^{(k-1)} $ and $ A^{(k)} $ are not the same, we must regenerate the preconditioner $ B^{(k)} $.}
\end{remark}

\begin{algorithm}[H]
	%	{\small{
			\caption{ASCPR-GMRES method}
			\label{alg:share-pgmres}
			\setcounter{AlgoLine}{0}
			\LinesNumbered
			\KwIn{$B^{(k-1)}, It^{(k-1)}, A^{(k)}, b^{(k)}, x_0, \mu, k, m, tol, MaxIt $;}
			\KwOut{$ x^{(k)}, It^{(k)}, B^{(k)} $;}		
			$ B^{(k)} $ = ASCPR($B^{(k-1)}, It^{(k-1)}, \mu, k$);
			%	$ B^{(k)} $ = ASCPR($A^{(k)}, B^{(k-1)}, b^{(k)}, It^{(k-1)}, \mu, k, x_0$);
			
			Compute $ r_0 \leftarrow b - A^{(k)}x_0 $, $ p_1 \leftarrow r_0 / \Vert r_0 \Vert $;
			
			\For{$It=1,\cdots, MaxIt$}
			{
				\For{$j=1,\cdots, m$}
				{
					Compute $ \bar{p} \leftarrow A^{(k)}(B^{(k)}p_j) $;
					
					Compute $ h_{i,j} \leftarrow (\bar{p}, p_i),i=1,\cdots, j$;
					
					Compute $ \tilde{p}_{j+1} \leftarrow \bar{p} - \sum\limits_{i=1}^{j}h_{i,j}p_i $;
					
					Compute $ h_{j+1,j} \leftarrow \Vert \tilde{p}_{j+1} \Vert$;
					
					\If{$h_{j+1,j}=0$}
					{
						$m \leftarrow j$; 				
						\textbf{break};
					}
					Compute $ p_{j+1} \leftarrow \tilde{p}_{j+1} / h_{j+1,j} $;
				}
				Solve the following minimization problem:
				$$ y_m = \arg\min\limits_{y \in \mathbb{R}^m} \Vert \beta e_1 - \bar{H}_m y \Vert, $$
				where $ \beta = \Vert p_1 \Vert $, $ e_1=(1,0,\cdots,0)^T \in \mathbb{R}^{m+1} $, $ \bar{H}_m := (h_{i,j}) \in \mathbb{R}^{(m+1) \times m} $;
				
				$ x_m \leftarrow x_0 + B^{(k)}(P_m y_m) $, here $ P_m :=(p_1, p_2, \cdots, p_m) $;
				
				Compute $ r_m \leftarrow b - A^{(k)}x_m $;
				
				\eIf{$\Vert r_m \Vert / \Vert r_0 \Vert < tol$}	
				{
					\textbf{break};
				} 
				{
					$ x_0 \leftarrow x_m$;
					
					$ p_1 \leftarrow r_m / \Vert r_m \Vert $;
				}
			}
			$ x^{(k)} \leftarrow x_m, It^{(k)} \leftarrow It $;
			
			\Return $ x^{(k)}, It^{(k)}, B^{(k)} $.
			%}}
\end{algorithm}

\begin{algorithm}[H]
	\caption{ASCPR method}
	\label{alg:share-setup}
	\setcounter{AlgoLine}{0}
	\LinesNumbered
	\KwIn{$B^{(k-1)}, It^{(k-1)}, \mu, k$;}
	\KwOut{ $ B^{(k)} $;}
	\eIf{$k>1$ and $ It^{(k-1)} \leq \mu $}
	{
		$ B^{(k)} \leftarrow B^{(k-1)} $;
	}
	{
		The preconditioner $ B^{(k)} $ is generated by calling \refalg{alg:CPR}.
	}
	\Return $ B^{(k)} $.
\end{algorithm}

We introduce a practical threshold $ \mu $ as a criterion for the adaptive SETUP preconditioner, which aims to improve the performance of the ASCPR-GMRES. To begin with, we explain the main idea of this approach. The $ It^{(k-1)} \leq \mu $ means $ B^{(k-1)} $ is an effective preconditioner for Jacobian system $ A^{(k-1)} x^{(k-1)} = b^ {(k-1)}$ because the smaller number of iterations $ It^{(k-1)}$, the more $ B^{(k-1)} $ approximates the inverse of matrix $ A^{(k-1)} $. Because the structure of these matrices is similar and the CPR preconditioner does not require high accuracy, the preconditioner $B^{(k-1)}$ can also be used as a preconditioner for Jacobian system $A^{(k)} x^{(k)} = b^{(k)}$. Moreover, the approach can improve the performance of the solver from the following two aspects. 
\begin{itemize}
	\item First, the efficiency of the solver is improved because the number of SETUP calls can be reduced. 
	
	\item Second, the parallel performance of the solver is improved because the proportion of low parallel speedup is reduced in the solver.
\end{itemize}
Choosing a suitable $\mu $ is important to the performance of the solver. Finally, we discuss the choice of $\mu $. If $ \mu $ is too small, the number of SETUP calls is not considerably reduced. As a result, the performance of the solver is not significantly improved. Specifically, when $\mu = 0$, ASCPR-GMRES degenerates into CPR-GMRES. On the contrary, if $ \mu $ is too large, the number of iterations of the solver is dramatically increased, thereby affecting the performance of the solver. Generally, the optimal $\mu$ is determined through numerical experiments according to concrete problems.

\section{A multi-color GS method}\label{sec:5}
In this section, we propose a parallel GS algorithm from the algebraic point of view, aiming to overcoming the limitations of the conventional red-black GS algorithm; it yields the same convergence behavior as the corresponding single-thread algorithm and obtains a good parallel speedup.

To this end, the concept of an adjacency graph is introduced. Note that an adjacency graph corresponds to a sparse matrix, and the nonzero entries of the matrix reflect the connectivity relationship between vertices in the graph. Assume that a sparse matrix $A \in \mathbb{R}^{n\times n} $ is symmetric. Let $G_{A}(V, E)$ be the (undirected) adjacency graph corresponding to the sparse matrix $A=\left(a_{ij}\right)_{n \times n}$, where $V=\left\{ v_{1}, v_{2}, \cdots, v_{n}\right\}$ is the vertices set, and $E =\left\{\left(v_{i}, v_{j}\right): \forall \, i \neq j, a_{ij} \neq 0 \right\}$ is the edges set (each nonzero entry $a_{ij}$ on the non diagonal of $A$  corresponds to an edge $(v_{i}, v_{j})$).

We are now in the position to give the design goals of this algorithm for grouping vertices and the parallel GS implementation based on the strong connections of $A$.

\subsection{Algorithm design goals}\label{sec:alg-target}
The goal of our algorithm design is to divide the vertices set $V$ into $c$ subsets $V_{1},V_{2},\cdots,V_{c}$ $(1 \leq c \leq n)$, and these subsets shall satisfy the following four conditions:
\begin{enumerate}
	\item [(a)] $V = V_{1} \cup V_{2} \cdots \cup V_{c}$;
	\item [(b)] $V_{i} \cap V_{j}=\varnothing,~i \neq j,~1 \leq i,j \leq c$;
	\item [(c)] Vertices in any subset are not connected, i.e., $a_{ij} =a_{ji}=0, \forall \, v_{i}, v_{j} \in V_{\ell}~(\ell=1,\cdots,c)$; and
	\item [(d)] The number of subsets $c$ should be as small as possible.
\end{enumerate}

It is easy to see that the smaller the number of groupings $c$, the more difficult the grouping, and the larger the parallel granularity. The classic GS is, in fact, equivalent to the situation when $c$ is equal to $n$.
\begin{remark}
	\rm{The red-black GS algorithm also satisfies the above four conditions with $c=2$.}
\end{remark}

\subsection{Parallel GS algorithm}
In 2003, \cite{2003Iterative} gave an upper bound estimate of the total number of colors based on the graph theory. To proceed, we briefly review this upper bound estimate.
\begin{proposition}[Upper bound estimation]\label{thm:1} 
	The number of multi-color groupings $c$ of undirected graph $G_{A}(V, E)$ does not exceed $\mathrm{degree}\big(G_{A}(V, E)\big)+1$; that is, $c$ does not exceed the maximum number of nonzero entries in each row of matrix $A$.
\end{proposition}

According to Proposition \ref{thm:1}, the number of groups $c$ depends on the maximum number of nonzero entries in each row of matrix $A$. When the matrix $A$ is relatively dense (such as the coarse grid matrix in AMG), the number of groups $c$ is large, which contradicts the condition (d) of Section \ref{sec:alg-target}. 
In extreme cases, the row nonzero entries of a matrix can be equal to the order of the matrix (this implies that $c=n$). Too many groups bring up two difficulties: low efficiency of the grouping algorithm and poor parallel performance. For the latter, because there are only small number of degrees of freedom within each group, it results in fine parallel granularity.

For this reason, $G_{A}(V, E)$ needs to be preprocessed to strengthen its sparseness before grouping the degrees of freedom. There are many ways to enhance the sparseness of $G_{A}(V, E)$. When the sparseness of $G_{A}(V, E)$ is enhanced, the number of groups $c$ becomes smaller. At the same time, the independence of vertex set $V_{\ell }~(\ell=1,\cdots,c)$ becomes worse, which affects the parallel results. Therefore, choosing an appropriate strategy is of great importance to enhance the sparseness of $G_{A}(V, E)$. The situation is considered where the solution vector is relatively smooth. It is found that the smaller nonzero entries (in the sense of absolute value) in each row of the matrix play a negligible role when a certain degree of freedom is smoothed by GS. In this paper, we propose a strategy to filter small nonzero entries. The matrix whose small nonzero entries are filtered is called the so-called ``matrix of strong connections." This concept is defined as follows.

The matrix $A=\left(a_{ij}\right)_{n \times n}$ corresponds to the matrix of strong connections $S(A, \theta)$ (denoted as $S$), and its entries are defined as
\begin{equation}\label{eq:strong-matrix}
	S_{i j}=\left\{\begin{array}{ll}
		1, & \left|a_{i j}\right|>\theta \sum\limits_{k=1}^{n}\left|a_{i k}\right| \\
		0, & \left|a_{i j}\right| \leq \theta \sum\limits_{k=1}^{n}\left|a_{i k}\right|
	\end{array} \quad \forall \, i, j=1,2, \cdots, n,~i \neq j, \right.
\end{equation}
where $\theta~(0 \leq\theta\leq1)$ is a given threshold, $S_{ij}=1$ when there is a strong adjacent edge between $v_{i}$ and $v_{j}$, and $S_{ij}=0$ means that there is no strong adjacent edge between $v_{i}$ and $v_{j}$.

To describe our algorithm conveniently, we introduce the following notations.
\begin{itemize}
	\item The set $S_{i}$ represents the vertex set that is strongly connected to the vertex $v_{i}$, i.e., $ S_i = \left\{ j: S_{ij} \neq 0,j=1,2, \cdots, n \right\} $.
	
	\item The set $\overline{S}_{i}$ represents the vertex set that is strongly connected to the vertex $v_{i}$ (including $v_{i}$) and whose colors are undetermined. That is, $ \overline{S}_i = \big\{ j: j\in S_{i} \cup \{i\}  ~\text{and}~ \text{color of}~j~\text{is} \\ \text{undetermined} \big\} $. 
	
	\item The set $\widehat{S}_{i}$ represents the vertex set that is the next most strongly connected to the vertex $v_{i}$ (the vertices on ``the second circle") and whose colors are undetermined. That is, $ \widehat{S}_i = \big\{ j: j\in W_{i}~\text{and}~\text{color of}~j~\text{is undetermined} \big\} $, where $ W_i = \big\{ j: \forall\,k \in S_i, j \in S_k / (S_i \cup \{i\}) \big\} $.
	
	\item The cardinality $ |\bigcdot|$ represents the number of entries in the set $ \bigcdot $. In particular, $\left|S_{i}\right|$ represents the influence value of the vertex $v_{i}$.	
\end{itemize}

\begin{figure}[htpb]
	\centering
	\includegraphics[width=0.25\linewidth]{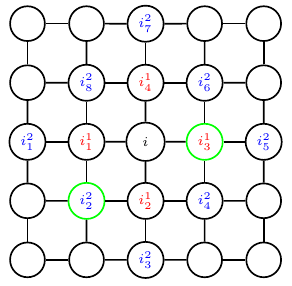}
	\caption{Schematic diagram of notations explanation.}
	\label{fig:notations}
\end{figure}

Let us explain these notations with a simple schematic diagram (see \reffig{fig:notations}). Assume that the all edges are strongly adjacent edges in \reffig{fig:notations}; moreover,  the vertices $ i^1_3 $ and $i^2_2$ are given the splitting attribute (i.e., they have their own color). Hence, $S_{i} = \left\{i^1_1, i^1_2, i^1_3, i^1_4\right\}$, $\overline{S}_{i} = \left\{i, i^1_1, i^1_2, i^1_3 \right\}$, 
$\widehat{S}_{i} = \left\{i^2_1, i^2_3, i^2_4, i^2_5, i^2_6, i^2_7, i^2_8 \right\}$, and $\left|S_{i}\right| = 4$.

In the following, our algorithms are presented. To begin, we propose a greedy splitting algorithm for the vertices set $ V $ based on the matrix of strong connections (denoted as VerticesSplitting); see \refalg{alg:strong-vsplit}. The proposed \refalg{alg:v-split} gives a vertices grouping algorithm corresponding to the matrix $ A $ (denoted as VerticesGrouping).

\begin{algorithm}[htpb]
	\caption{VerticesSplitting method} %????
	\label{alg:strong-vsplit}
	\setcounter{AlgoLine}{0}
	\LinesNumbered
	\KwIn{$V,~S$;}%????
	\KwOut{$ W,~\overline{W}$;}%??
	
	Set $ W \leftarrow \varnothing,~\overline{W} \leftarrow \varnothing,~ \widehat{W} \leftarrow \varnothing $;
	
	\While{$ V \neq \varnothing $}
	{
		\eIf{$ \widehat{W} \neq \varnothing $}{
			Any take $ v_i \in \widehat{W}$ and $ | S_i | \geq  | S_j |$, $ \forall \, v_i, v_j \in \widehat{W} $;
		}{
			Any take $ v_i \in V$ and $ | S_i | \geq  | S_j |$, $ \forall \, v_i, v_j \in V $;
		}
		\uIf{$ v_i $ is not strongly connected to any vertices in the set $ W $ (i.e., $S_{ij} = 0,\forall \, j \in W$)}	
		{	
			$ W \leftarrow W \cup v_i,~V \leftarrow V/v_i $;
			
			\If{$ v_i \in \widehat{W} $}
			{
				$ \widehat{W} \leftarrow \widehat{W}/v_i $;
			}		
			$ \overline{W} \leftarrow \overline{W} \cup S_i $;		
			
			$ V \leftarrow V / S_i $;		
			
			$ \widehat{W} \leftarrow \widehat{W} \cup \widehat{S}_i $;
		}
		\Else{
			$ \overline{W} \leftarrow \overline{W} \cup v_i $;
			
			$ V \leftarrow V / v_i $;
			
			\If{$ v_i \in \widehat{W} $}
			{
				$ \widehat{W} \leftarrow \widehat{W}/v_i $;
			}
		}
	}
	\Return $ W,~\overline{W}$.
\end{algorithm}

\begin{algorithm}[htpb]
	\caption{VerticesGrouping method}%????
	\label{alg:v-split}
	\setcounter{AlgoLine}{0}
	\LinesNumbered
	\KwIn{$V,~S$;}%????
	\KwOut{$ V_\ell~(\ell=1,\cdots,c)$;}%??
	
	Set $c \leftarrow 0$; 
	
	\While{$ V \neq \varnothing$}
	{
		$c \leftarrow c + 1$;
		
		Call \refalg{alg:strong-vsplit} to generate $V_c$ and $\overline{V}_c$;
		
		Let $ V \leftarrow \overline{V}_c$;
	}
	\Return $ V_\ell~(\ell=1,\cdots,c)$.
\end{algorithm}

As can be easily noticed from \refalgs{alg:strong-vsplit} and \ref{alg:v-split}, the grouping numbers $c$ and the degree of independence of the vertices set $V_\ell~(\ell=1,\cdots,c)$ depend on the choice of strength threshold $\theta$. The smaller the value of $\theta$, the better the degree of independence of the vertices set, but the greater the $c$. Especially, when $ \theta=0 $, the degree of independence of the vertices set is the best (complete independence), but $ c $ is the largest. Moreover, when $\theta=1$, the degree of independence of the vertices set is the worst, and $c=1$ (at this time, the proposed algorithm degenerates to the classic GS algorithm). In this sense, it is necessary to balance the grouping numbers and the degree of independence within the vertices set. In order to better satisfy condition (d), we use an approach to weaken the condition (c) slightly in our algorithms. We expect that this approach can slightly improve the parallel performance.

Next, two properties of our proposed algorithms are given.
\begin{proposition}[Matrix diagonalization]\label{property:1}
	The block matrix $ A_{V_\ell V_\ell} $ (with $ V_\ell $ as the row and column indices, $\ell=1,\cdots,c$) is the diagonal matrix if the row and column indices of matrix $ A $ are rearranged by the indices set $ \big\{ V_\ell \big\}_{\ell=1}^{c}$. 
\end{proposition}

\begin{proposition}[Finite termination]\label{property:2}
	\refalg{alg:v-split} terminates within a finite number of steps; that is, $c \leq |V|$.
\end{proposition}
\begin{proof}
	To prove that \refalg{alg:v-split} terminates within a finite step, just prove $ V_c \neq \varnothing $ in line 5 of \refalg{alg:v-split}. According to \refalg{alg:strong-vsplit}, if $V \neq \varnothing$, $V_c$ (i.e., the output variable $W$ of \refalg{alg:strong-vsplit}) contains at least one vertex. From lines 3--7 of \refalg{alg:v-split}, $ c \leq |V|$ can be obtained.
\end{proof}

Note that the \refalg{alg:v-split} can split $G_{A}(V, E)$ into subgraphs $G_{A_\ell}(V_\ell, E_\ell), \ell=1,\cdots,c$. The $ S(A_\ell, \theta) $ is denoted as the adjacency matrix corresponding to each subgraph. Moreover, each subgraph corresponds to a submatrix $A_\ell(V_\ell)$ of matrix $ A $. According to Proposition \ref{property:1}, the diagonal blocks of the submatrix $A_\ell(V_\ell)$ are diagonal, and the GS smoothing of the submatrix $A_\ell(V_\ell)$ is completely parallel at this time. Furthermore, a parallel (or multi-color) GS algorithm based on the strong connections of the matrix is given by \refalg{alg:GS-parallel}, denoted as PGS-SCM.

\begin{algorithm}[H]
	\caption{PGS-SCM method}
	\label{alg:GS-parallel}
	\setcounter{AlgoLine}{0}
	\LinesNumbered
	\KwIn{$ A, x, b, \theta $;}
	\KwOut{$ x $;}
	
	Using the matrix $ A $ and the formula \eqref{eq:strong-matrix} to generate  vertices set $ V $ and matrix of strong connections $ S $, respectively;
	
	Call \refalg{alg:v-split} to generate independent vertices subset $V_{\ell}$ ($\ell=1,\cdots,c$);
	
	Using $V_{\ell}$ to split the matrix $ A $ into submatrix $ A_\ell$ ($\ell=1,\cdots,c$);
	
	\For{$\ell=1, \cdots, c$}
	{
		Parallel call the classic GS algorithm of submatrix $A_\ell$;
	}
	
	\Return $ x $.
\end{algorithm}

Finally, the proposed algorithms can be parallelized for multi-core and many-core architectures. We develop the OpenMP version and the CUDA version of the parallel program, respectively. Furthermore, these algorithms are integrated into the FASP framework. The results of the numerical experiments will be given in the next section.

\section{Numerical experiments}\label{sec:6}
In this section, to demonstrate the performance of the proposed methods, we consider the two-phase and three-phase test problems based on the SPE10 benchmark. The numerical experiments are performed on a machine with Intel Xeon Platinum 8260 CPU (32 cores, 2.40GHz), 128GB DRAM, and NVIDIA Tesla T4 GPU (2560 cores, 16GB Memory).

Here we provide some details of the ASCPR-GMRES method. For the CPR preconditioner, in the first stage, we use the Unsmoothed Aggregation AMG (UA-AMG) method \cite{1993Multi} to approximate the inverse of the pressure coefficient matrix, where the aggregation strategy is the so-called non-symmetric pairwise matching aggregation (NPAIR) \cite{NapovPariwise}, the cycle type is the Nonlinear AMLI-cycle \cite{HuXiaozhe2013}, the smoothing operator is PGS-SCM, the degree of freedom of the coarsest space is set to be 10000, and the coarsest space solver is a direct solver. In the second stage, we use the BILU method based on Level Scheduling (LS) \cite{2003Iterative} to approximate the inverse of the coefficient matrix. For the restarted GMRES method, the restarting number $m$ is 28, the maximum number of iterations \emph{MaxIt} is 100, and the tolerance error of relative residual norm \emph{tol} is $10^{-5}$.

\subsection{Two-phase SPE10}
The standard two-phase SPE10~\cite{2001TenthSPE} benchmark is tested to demonstrate the performance of the proposed methods. The model dimensions are $1200 \times 2200 \times 170$ (ft), and the number of grid cells is $60\times220\times80$ (the total number of grid cells is 1,122,000 and the number of active cells is 1,094,422). First, we verify convergence behavior and parallel performance of the PGS-SCM method. Furthermore, we test the parallel performance for the ASCPR-GMRES-OMP and ASCPR-GMRES-CUDA methods, where the ASCPR-GMRES-OMP and ASCPR-GMRES-CUDA correspond to solver versions in OpenMP and CUDA, respectively. Finally, our results based on the in-house simulator of PetroChina, HiSim 2.0~\cite{2014A1}, are compared with the results obtained using a commercial simulator, tNavigator (2020 version)~\cite{tNavigatorManual2020}.

\subsubsection{PGS-SCM method}
To evaluate the convergence behavior and parallel performance of the proposed PGS-SCM method, we employ the CPR-GMRES method as the solver for the petroleum reservoir simulation. Furthermore, we employ the parallel GS method based on natural ordering (denoted as PGS-NO) as a reference for comparison.
\begin{example}\label{ex:1-2P}
	\rm{The two-phase example is considered, and the numerical simulation conducted for 2000 days. We use the different number of threads ($ \mathrm{NT}=1,2,4,8$, and $16$) to test the parallel performance of CPR-GMRES-PGS-NO and CPR-GMRES-PGS-SCM. The impacts of the different strength thresholds $ \theta $ ($\theta=0,0.05,0.1$, and $0.3 $) on CPR-GMRES-PGS-SCM are also tested.}
\end{example}
\reftab{tab:MC-GS-2P} lists the total number of linear iterations (\emph{Iter}), the total wall time in seconds (\emph{Time}), and the parallel speedup (\emph{Speedup} defined in Remark \ref{rmk:Speedup}).

\begin{table}[h]
	\centering
	\setlength{\tabcolsep}{4.5pt} 
	\renewcommand\arraystretch{1.5} 
	\caption{Iter, Time(s), and Speedup of the two solvers with different NT for the two-phase SPE10 problem.}
	\label{tab:MC-GS-2P}
	\begin{tabular}{c|c|c|ccccc}
		\toprule%[1.2pt]
		Solvers & $\theta$ & NT  & 1& 2 & 4 & 8 & 16 \\
		\hline
		\multirow{3}{*}{CPR-GMRES-PGS-NO}
		&\multirow{3}{*}{---}	&\emph{Iter}	  &5823 	&5826 	&5842 	&5882 	&5920 \\
		&											&\emph{Time}		&4701.26 	&2577.35 	&1497.12 	&988.09 	&890.78 \\
		&											&\emph{Speedup} &1.00 	&1.82 	&3.14 	&4.76 	&5.28 \\	
		\hline
		\multirow{12}{*}{CPR-GMRES-PGS-SCM}
		&\multirow{3}{*}{0}		&\emph{Iter}	  &\textbf{5837} 	&\textbf{5837} 	&\textbf{5837} 	&\textbf{5837 }	&\textbf{5837} \\
		&											&\emph{Time}		&4780.42 	&2610.02 	&1501.12 	&976.80 	&847.26 \\
		&											&\emph{Speedup} &1.00 	&1.83 	&3.18 	&4.89 	&5.64 \\
		\cline{2-8}	
		&\multirow{3}{*}{0.05}&\emph{Iter}	  &5823 	&5822 	&5818 	&5822 	&\textbf{5827} \\
		&											&\emph{Time}		&4753.00 	&2593.07 	&1491.97 	&970.26 	&\textbf{829.24} \\
		&											&\emph{Speedup} &1.00 	&1.83 	&3.19 	&4.90 	&\textbf{5.73} \\	
		\cline{2-8}	
		&\multirow{3}{*}{0.1} &\emph{Iter}	  &5832 	&5833 	&5826 	&5825 	&5846 \\
		&											&\emph{Time}		&4782.52 	&2599.12 	&1497.25 	&977.93 	&850.82 \\
		&											&\emph{Speedup} &1.00 	&1.84 	&3.19 	&4.89 	&5.62 \\
		\cline{2-8}	
		&\multirow{3}{*}{0.3}	&\emph{Iter}	  &5821 	&5839 	&5841 	&5876 	&5912 \\
		&											&\emph{Time}		&4732.72 	&2615.05 	&1510.71 	&981.28 	&872.32 \\
		&											&\emph{Speedup} &1.00 	&1.81 	&3.13 	&4.82 	&5.43 \\		
		\bottomrule
	\end{tabular}
\end{table}
\begin{remark}\label{rmk:Speedup}
\rm{The calculation formula of speedup is
	$$ Speedup = \frac{T_1}{T_n}, $$
	where $ T_1 $ represents the wall time obtained by a single thread (core), and $ T_n $ represents the wall time obtained by $n$ threads.}
\end{remark}

It can be seen from \reftab{tab:MC-GS-2P} that, for CPR-GMRES-PGS-NO, the total number of linear iterations gradually increases as NT increases. In particular, if $\mathrm{NT}=16$, the total number of linear iterations increases by 97 compared with the single-thread case (the speedup is about 5.28). On the other hand, the number of iterations of the parallel GS algorithm based on natural ordering is not as stable. That is, as the number of threads increases, the number of iterations increases, which affects the parallel speedup as well. 
However, for CPR-GMRES-PGS-SCM (when $ \theta=0 $), the total number of linear iterations is not changed with a larger NT. Such results indicate that the proposed PGS-SCM method yields same convergence behavior as the corresponding single-thread algorithm.  
Moreover, from the perspective of parallel performance, CPR-GMRES-PGS-SCM gets a higher speedup compared with CPR-GMRES-PGS-NO. For example, when $\mathrm{NT}=16$, the speedup of CPR-GMRES-PGS-SCM is also higher (the speedup is about 5.73).  

Next, we discuss the influence of the strength threshold $ \theta $ on the parallel performance for the CPR-GMRES-PGS-SCM solver.
When $ \theta $ is small (for example, $\theta=0.0$ or $0.05$), the total number of linear iterations changes little when NT increases. However, if $\theta$ gets larger (for example, $\theta=0.1$ or $0.3 $), the varied range of the total number of linear iterations is enlarged with the increase of NT. If $\mathrm{NT}=16$ is considered, as $ \theta $ increases, the total number of linear iterations first decreases and then increases, and the speedup first increases and then decreases. In particular, when $ \theta = 0.05$, the minimum total number of linear iterations is 5827, and the speedup is the highest. This shows that the strength threshold $ \theta $ affects the convergence as well as the parallel performance.

\subsubsection{ASCPR-GMRES-OMP method}
\begin{example}\label{ex:5}
	\rm{For the two-phase SPE10 example, the numerical simulation is conducted for 2000 days, and $ \theta=0.05 $. We explore the parallel performance of ASCPR-GMRES-OMP for four different values of $ \mu $ (i.e., $ \mu = 0, 20, 30$, and $40$), when $\mathrm{NT} = 1, 2, 4, 8$, and $16$, respectively.}
\end{example}
	
\begin{table}[h]
	\centering
	\setlength{\tabcolsep}{14pt} 
	\renewcommand\arraystretch{1.5} % ???
	\caption{SetupCalls, SetupRatio, Iter, Time(s), $\text{{Speedup}}^{*}$, and Speedup for the two-phase SPE10 problem.}
	\label{tab:Time-Newton-Linear-SPE10-2P}	
	\begin{tabular}{cccccccc}
		\toprule
		& $\mu$  & 1& 2 & 4 & 8 & 16\\
		\hline
		\multirow{4}{*}{\emph{SetupCalls}}  
		&0	&239	&239	&239	&239	&239	\\
		&20	&188	&188	&188	&188	&188	\\
		&30	&58		&58		&58		&58		&58	\\
		&40	&33		&33		&33		&33		&33	\\
		\hline
		\multirow{4}{*}{\emph{SetupRatio}}  
		&0  &11.08\%	&14.59\%	&20.63\%	&29.34\%	&41.38\% \\
		&20	& 9.45\%	&12.46\%	&17.55\%	&25.45\%	&38.14\% \\
		&30	& 5.19\%	& 6.52\%	& 9.20\%	&14.05\%	&26.09\% \\
		&40	& 4.05\%	& 4.99\%	& 6.62\%	&11.03\%	&21.83\% \\
		\hline
		\multirow{4}{*}{\emph{Iter}}  
		&0	&5823	&5822	&5818	&5822	&5827 \\
		&20	&5855	&5854	&5853	&5855	&5857 \\
		&30	&6309	&6308	&6315	&6317	&6319 \\
		&40	&7033	&7034	&7037	&7041	&7044 \\
		\hline
		\multirow{4}{*}{\emph{Time}}
		&0	&4753.00 	&2593.07 	&1491.97 	&970.26 	&829.24 \\
		&20	&4821.98 	&2606.18 	&1488.95 	&949.98 	&804.56 \\
		&\textbf{30}	&\textbf{4919.8}1 	&\textbf{2617.66} 	&\textbf{1444.92} 	&\textbf{877.85} 	&\textbf{718.13} \\
		&40	&5475.72 	&2880.97 	&1570.72 	&946.35 	&769.34 \\
		\hline
		\multirow{4}{*}{$\text{\emph{Speedup}}^{*}$}  
		&0	&1.00 	&1.83 	&3.19 	&4.90 	&5.73 \\
		&20	&0.99 	&1.82 	&3.19 	&5.00 	&5.91 \\
		&\textbf{30}	&\textbf{0.97} 	&\textbf{1.82} 	&\textbf{3.29} 	&\textbf{5.4}1 	&\textbf{6.62} \\
		&40	&0.87 	&1.65 	&3.03 	&5.02 	&6.18 \\
		\hline
		\multirow{4}{*}{$\emph{Speedup}$}  
		&0	&1.00 	&1.83 	&3.19 	&4.90 	&5.73 \\
		&20	&1.00 	&1.85 	&3.24 	&5.08 	&5.99 \\
		&30	&1.00 	&1.88 	&3.40 	&5.60 	&6.85 \\
		&40	&1.00 	&1.90 	&3.49 	&5.79 	&7.12 \\
		\bottomrule
	\end{tabular}
\end{table}

To assess the effects of different thresholds $\mu $ on the performance of ASCPR-GMRES-OMP for the two-phase SPE10 problem, \reftab{tab:Time-Newton-Linear-SPE10-2P} lists the number of SETUP calls (\emph{SetupCalls}), the ratio of SETUP in the total solution time (\emph{SetupRatio}), the total number of linear iterations (\emph{Iter}), total solution time in seconds (\emph{Time}), new parallel speedup ($\text{\emph{Speedup}}^{*}$, see Remark \ref{rmk:Speedup-star}), and parallel speedup (\emph{Speedup}).
	
\begin{remark}\label{rmk:Speedup-star}
\rm{In order to compare the parallel speedup of ASCPR-GMRES-OMP fairly, we propose a new parallel speedup (denoted as {$\text{\emph{Speedup}}^{*}$}), and it is defined as follows:
$$ \text{\emph{Speedup}}^{*} = \frac{T_1^{0}}{T_n^{\mu}}, $$
where $ T_1^{0} $ represents the wall time obtained by a single thread when the general preconditioner ($ \mu=0 $) is used, and $ T_n^{\mu} $ represents the wall time obtained by $ n $ threads when the adaptive SETUP threshold $ \mu $ is used.}
\end{remark}

From~\reftab{tab:Time-Newton-Linear-SPE10-2P}, we first look at the number of SETUP calls, the ratio of SETUP in the total solution time, and the total number of linear iterations. When NT is fixed, both the number of SETUP calls and the ratio of SETUP in the total solution time decrease as $\mu$ increases. In particular, when $\mu=40$, there are only 33 SETUP calls. When $\mu$ is fixed, the number of SETUP calls does not changed with respect to NT, but the ratio of SETUP in the total solution time is increasing. If NT is fixed, the total number of linear iterations is increasing as $\mu$ increases. These observations indicate that the number of SETUP calls is significantly decreased with the increase of $\mu$, while the total number of linear iterations is also increased. Furthermore, we focus on the total solution time and the new parallel speedup. When $ \mathrm{NT}=16 $ and $\mu=30$, the total solution time is 718.13 (s), and the corresponding the new parallel speedup is 6.62. Compared with the regular CPR preconditioner, the total solution time of ASCPR-GMRES-OMP is reduced from 829.24 (s) to 718.13 (s), and the new parallel speedup is increased from 5.73 to 6.62. These results show that the proposed method can improve the parallel performance of the solver.

\subsubsection{ASCPR-GMRES-CUDA method}
\begin{example}\label{ex:6}
	\rm{For the two-phase SPE10 example, the numerical simulation is conducted for 2000 days, and $ \theta=0.05 $. We explore the impacts of $ \mu $ ($ \mu = 0, 20, 30$, and $40$) on the parallel performance of ASCPR-GMRES-CUDA.}
\end{example}

\reftab{tab:GPU-results-2p} lists \emph{SetupCalls}, \emph{SetupRatio}, \emph{Iter}, \emph{Time}, and \emph{Speedup} of ASCPR-GMRES-CUDA for the two-phase SPE10 problem. The single-thread calculation results of the OpenMP version program ASCPR-GMRES-OMP(1) are also added for comparison. For ASCPR-GMRES-CUDA solver, as $ \mu $ increases, both the number of SETUP calls and the ratio of SETUP in the total solution time decreases, the total number of linear iterations gradually increases, and the parallel speedup increases. The CUDA version can be 6.22 times faster than the ASCPR-GMRES-OMP(1). In case $ \mu=40 $, the total solution time of ASCPR-GMRES-CUDA is reduced from 763.66 (s) to 424.60 (s), which can be 1.80 times faster compared with $ \mu=0 $. At the same time, compared with ASCPR-GMRES-OMP(1), the speedup of ASCPR-GMRES-CUDA reaches 11.19.

\begin{table}[h]
	\centering
	\setlength{\tabcolsep}{7.0pt} 
	\renewcommand\arraystretch{1.5} % ???
	\caption{SetupCalls, SetupRatio, Iter, Time(s), and Speedup (compared to ASCPR-GMRES-OMP(1)) of the different $\mu$ for the two-phase SPE10 problem.}
	\label{tab:GPU-results-2p}
	\begin{tabular}{c|c|c|c|c|c|c}
		\hline
		Solvers & $\mu$& \emph{SetupCalls}&	\emph{SetupRatio} & \emph{Iter}&	\emph{Time}&	\emph{Speedup} \\
		\hline
		ASCPR-GMRES-OMP(1) &0	&239	&12.22\%	&5823	&4753.00 	& ---   \\
		\hline
		\multirow{5}{*}{ASCPR-GMRES-CUDA}
		&0	&239	&71.39\%	&5525	&763.66 	&6.22  \\
		&20	&185	&66.11\%	&5659	&677.13 	&7.02  \\
		&30	&52		&41.99\%	&6182	&432.45 	&10.99 \\
		&\textbf{40}	&\textbf{34}		&\textbf{34.58\%}	&\textbf{6740}	&\textbf{424.60} 	&\textbf{11.19} \\
		\hline
	\end{tabular}
\end{table}

Finally, we present the effects of different $\mu$ on the timestep size $ \Delta t $ for ASCPR-GMRES-OMP and ASCPR-GMRES-CUDA solvers, see~\reffig{fig:DeltaT-SPE10-2P}. It can be seen from~\reffig{fig:DeltaT-SPE10-2P-OMP-MU} that when $\mu=0$, the timestep size is consistent for ASCPR-GMRES-OMP with different thread numbers. From~\reffigs{fig:DeltaT-SPE10-2P-OMP-NT} and~\ref{fig:DeltaT-SPE10-2P-GPU}, when $\mu$ changes, the timestep sizes of the ASCPR-GMRES-OMP and ASCPR-GMRES-CUDA solvers rarely change, and the effect on the overall numerical results can be ignored.

\begin{figure}
	\centering
	\subfigure[OpenMP, $ \mu=0 $]{\label{fig:DeltaT-SPE10-2P-OMP-MU}
		\begin{minipage}[t]{0.33\linewidth}
			\centering
			\includegraphics[width=5.5cm]{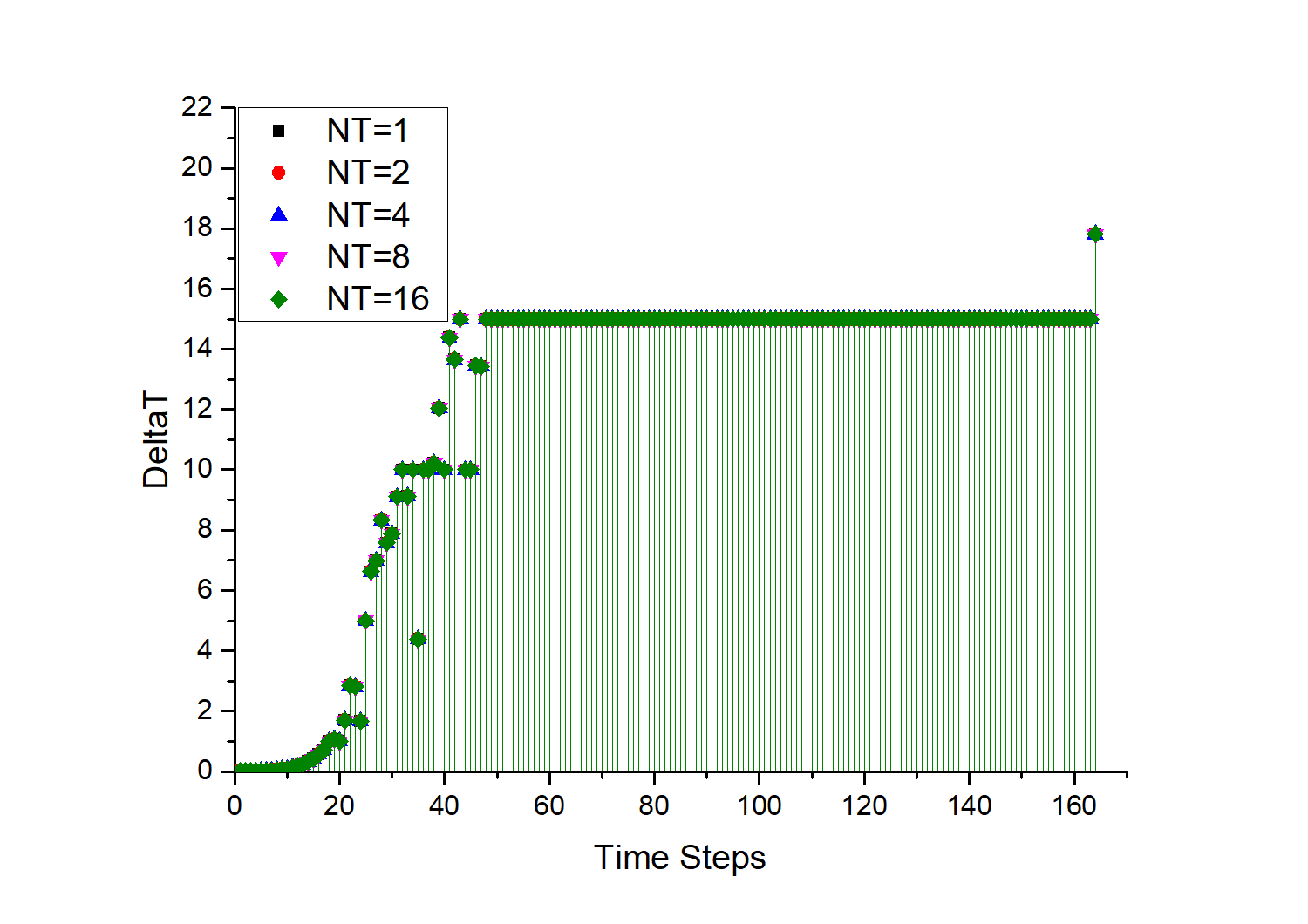}
		\end{minipage}
	}%	
	\subfigure[OpenMP, $ NT=16 $]{\label{fig:DeltaT-SPE10-2P-OMP-NT}
		\begin{minipage}[t]{0.33\linewidth}
			\centering
			\includegraphics[width=5.5cm]{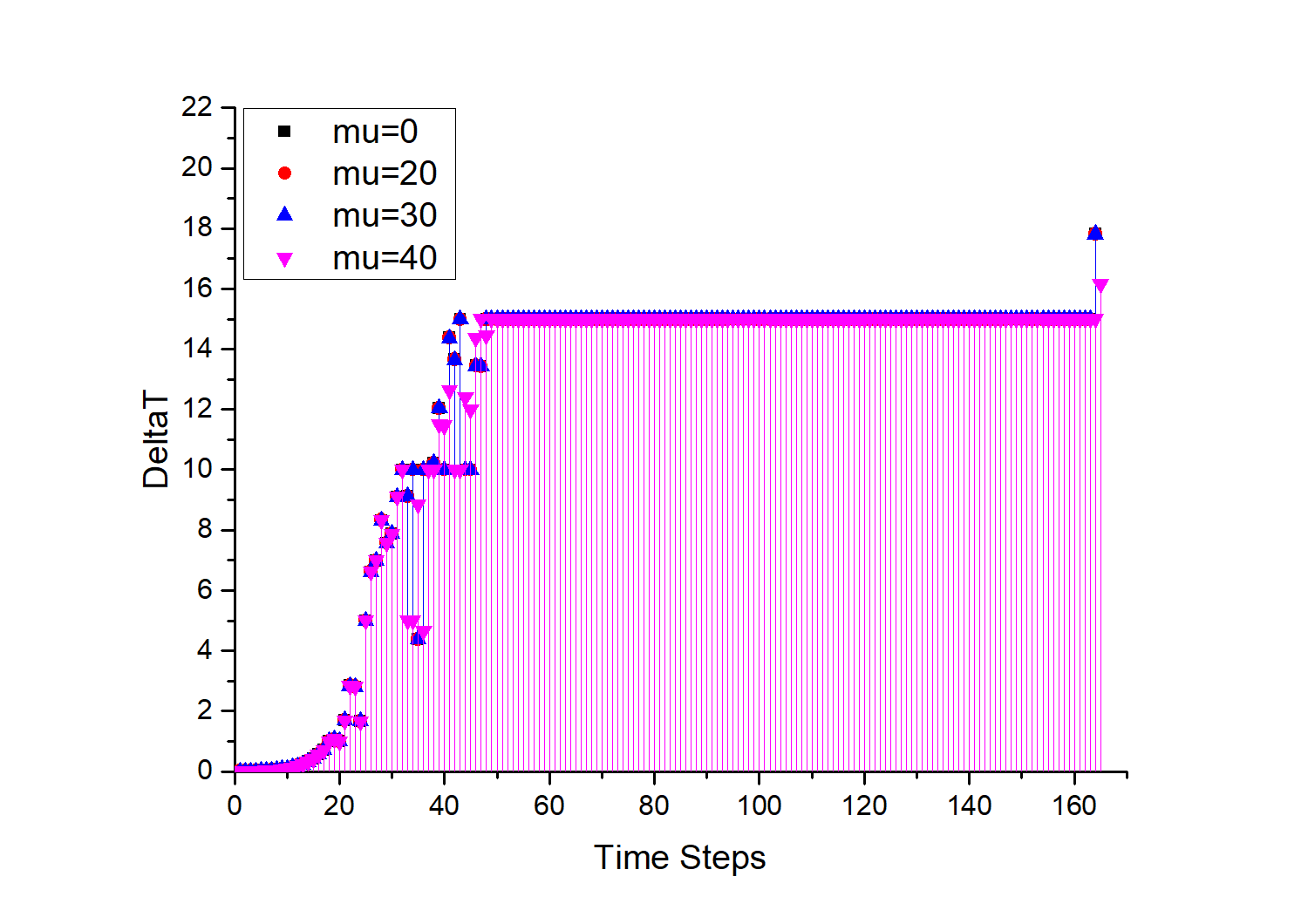}
		\end{minipage}
	}%
	\subfigure[GPU]{\label{fig:DeltaT-SPE10-2P-GPU}
		\begin{minipage}[t]{0.33\linewidth}
			\centering
			\includegraphics[width=5.5cm]{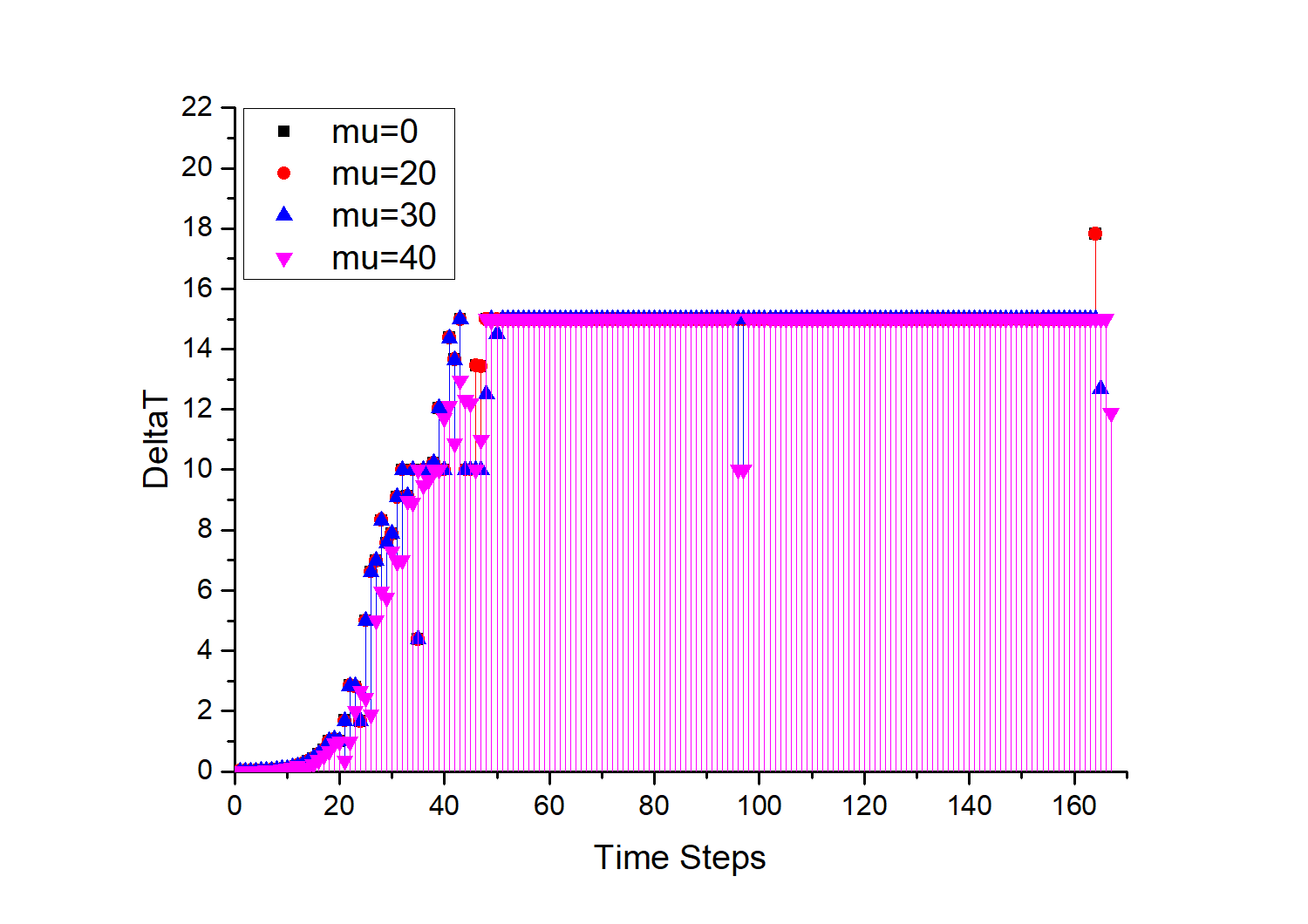}
		\end{minipage}
	}%
	\caption{The timestep size $ \Delta t $ of ASCPR-GMRES-OMP and ASCPR-GMRES-CUDA for the two-phase SPE10 problem.}\label{fig:DeltaT-SPE10-2P}
\end{figure}

\subsubsection{Performance comparisons with commercial simulator} 
To better evaluate the performance of the proposed methods, we also test the same problem with a commercial simulator for comparison. The default solving method and parameters are used, where the maximum number of iterations \emph{MaxIt} is 1000 and the tolerance of relative residual norm \emph{tol} is set to $10^{-5}$.  We compare the experimental results of the OpenMP and GPU versions for commercial and our simulators, respectively. 

To begin with, \reffigs{fig:FOPR-SPE10-2P} and \ref{fig:FPR-SPE10-2P} display the field oil production rate and average pressure graphs. Through quantitative comparison, it can be found that our simulation results are consistent with the results of commercial simulator. 

\begin{figure}[H]
	\centering	
	\subfigure[OpenMP, $ NT=16 $]{
		\begin{minipage}[t]{0.5\linewidth}
			\centering
			\includegraphics[width=6cm]{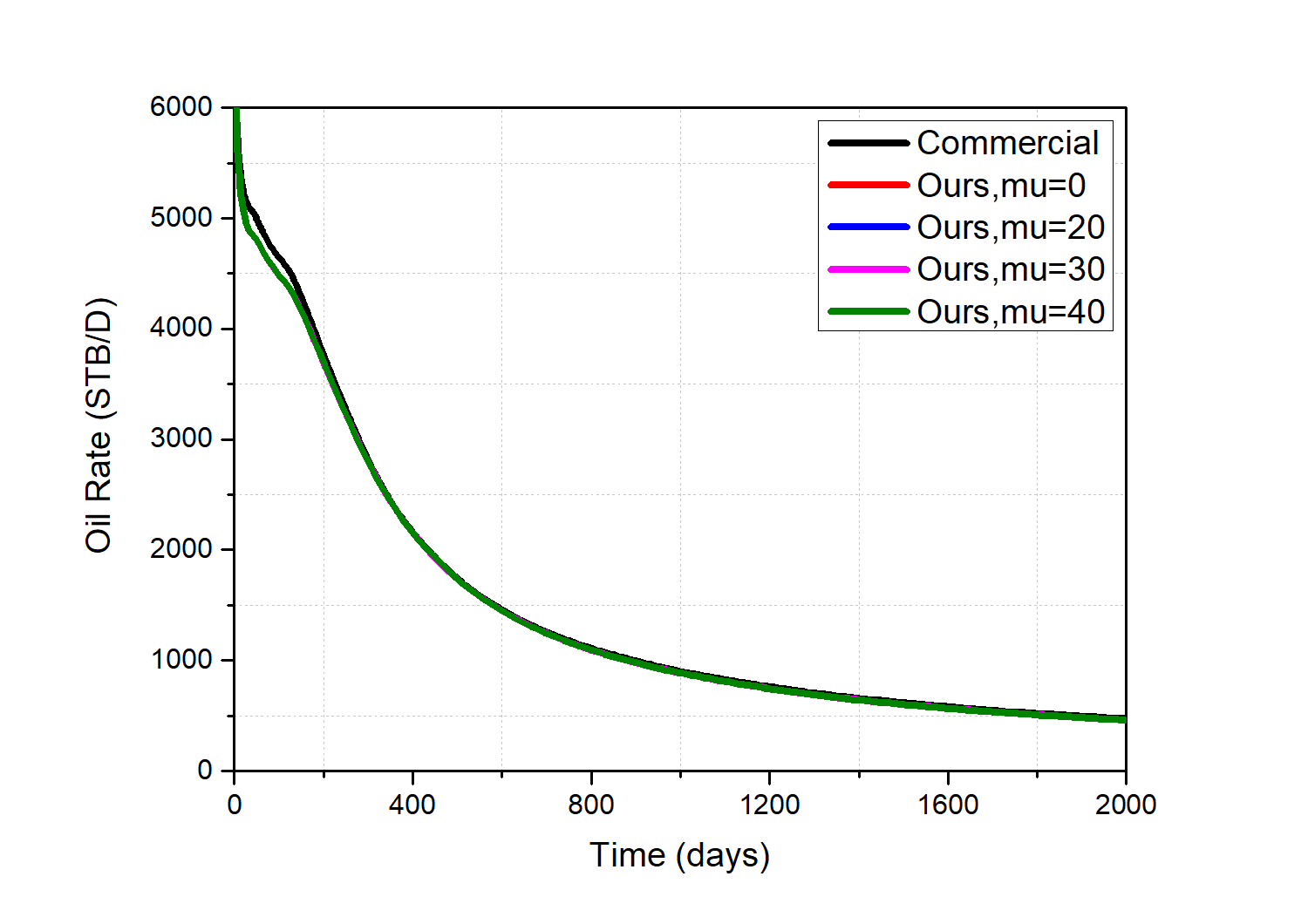}
		\end{minipage}
	}%
	\subfigure[GPU]{
		\begin{minipage}[t]{0.5\linewidth}
			\centering
			\includegraphics[width=6cm]{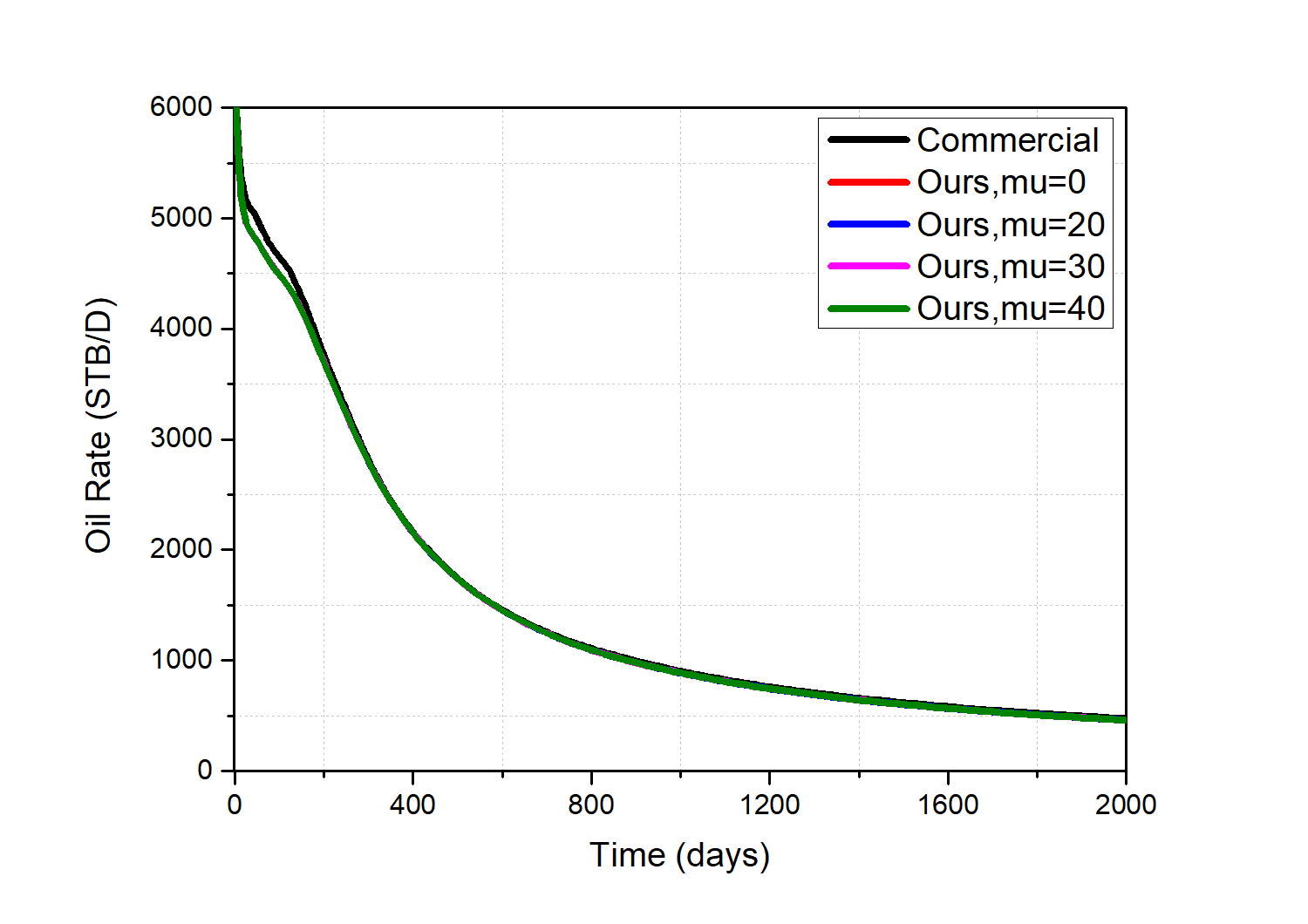}
		\end{minipage}
	}%
	\caption{Field oil production rate of OpenMP and GPU for the two-phase SPE10 problem.}\label{fig:FOPR-SPE10-2P}
\end{figure}
\begin{figure}[h]
	\centering	
	\subfigure[OpenMP, $ NT=16 $]{
		\begin{minipage}[t]{0.5\linewidth}
			\centering
			\includegraphics[width=6cm]{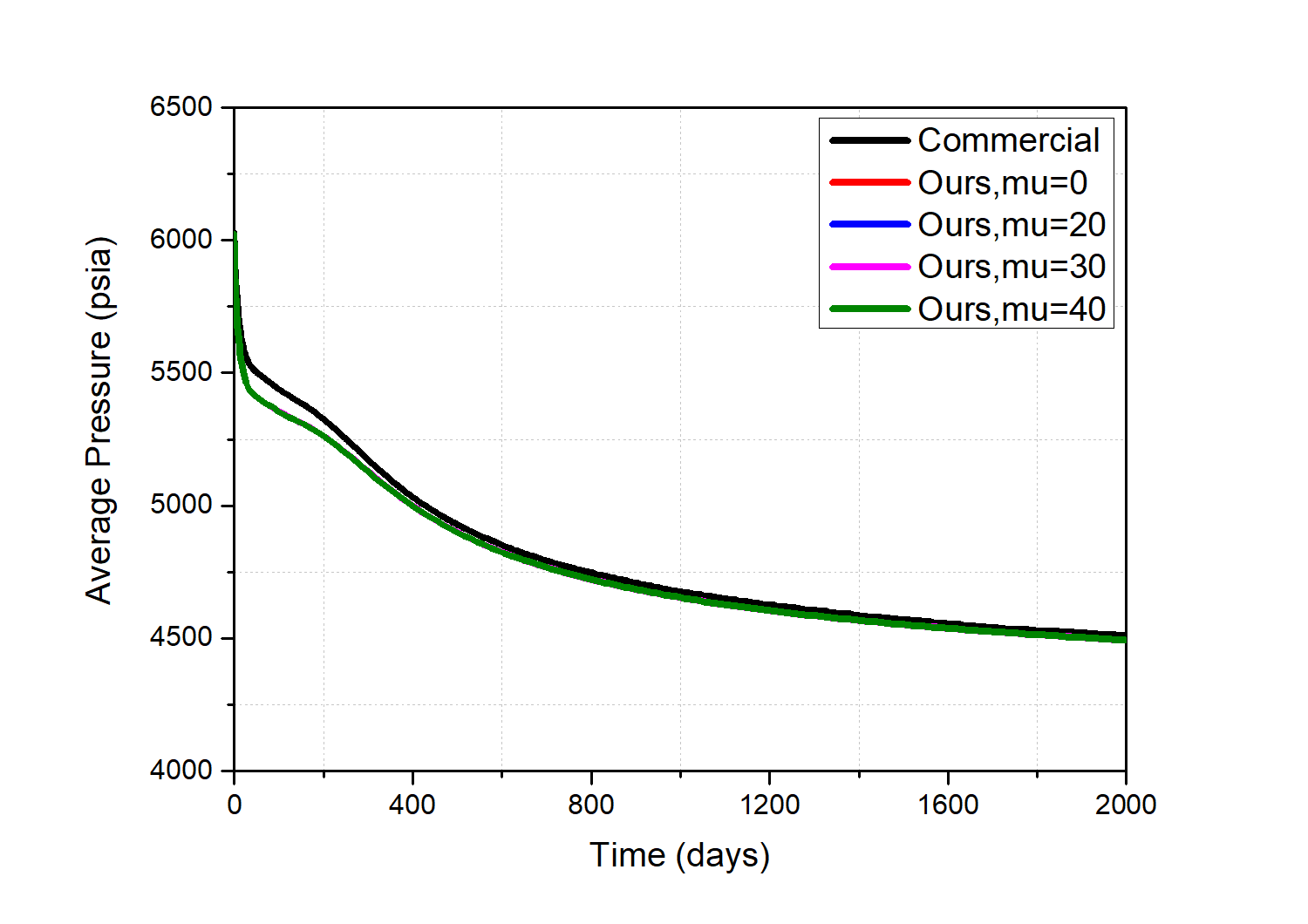}
		\end{minipage}
	}%
	\subfigure[GPU]{
		\begin{minipage}[t]{0.5\linewidth}
			\centering
			\includegraphics[width=6cm]{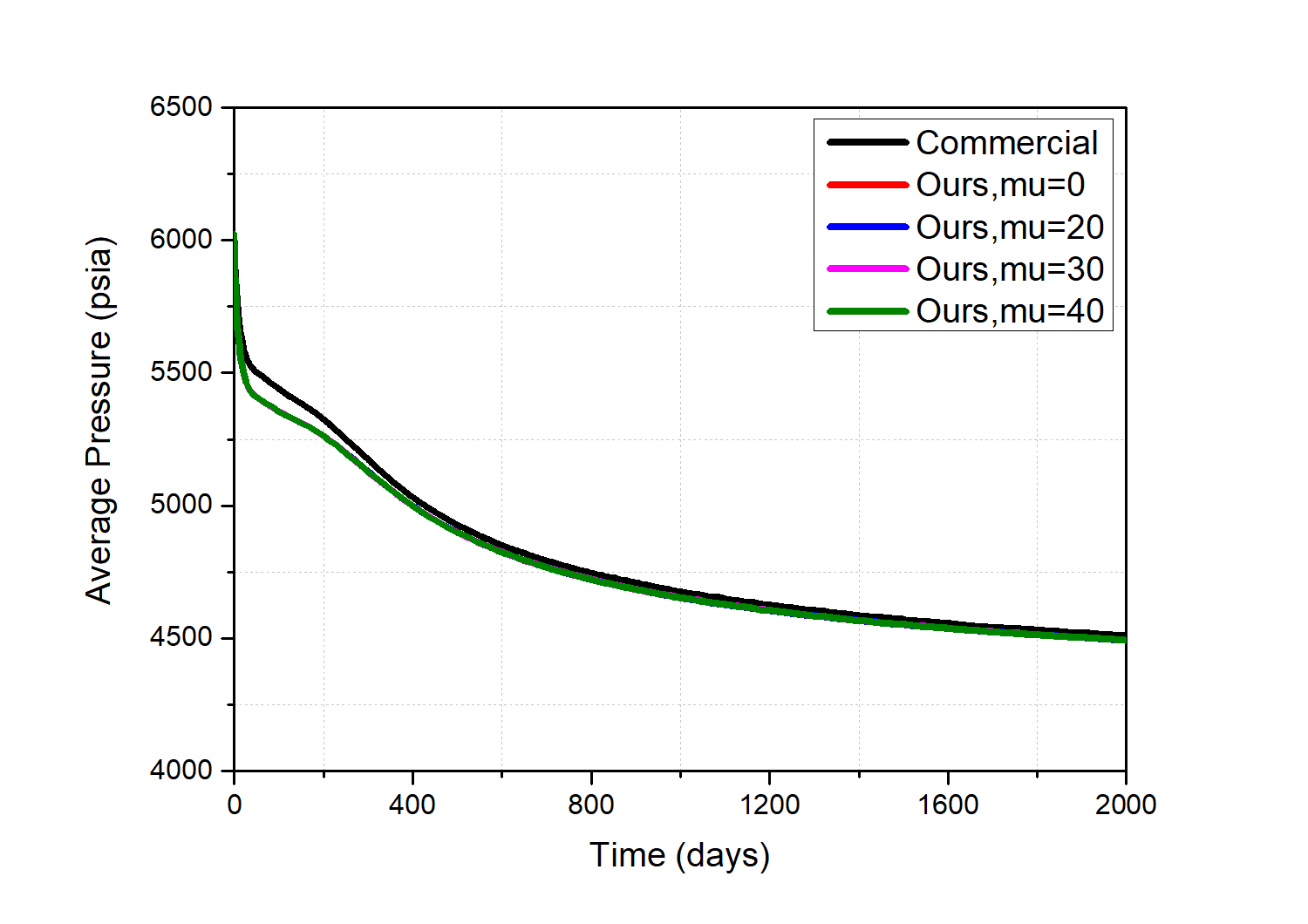}
		\end{minipage}
	}%
	\caption{Average pressure of OpenMP and GPU for the two-phase SPE10 problem.}\label{fig:FPR-SPE10-2P}
\end{figure}

\begin{table}[H]
	\centering
	\setlength{\tabcolsep}{10pt} 
	\renewcommand\arraystretch{1.2} 
	\caption{NumTSteps, NumNSteps, Iter, AvgIter, Time(h), and Speedup comparisons of the OpenMP version of commercial and our simulators for the two-phase SPE10 problem.}
	\label{tab:tNavigator-vs-HiSim-OMP-SPE10-2P}		
	\begin{tabular}{c|c|cccccc}
		\toprule
		Simulators & $\mu$ & NT & 1& 2 & 4 & 8 & 16\\
		\hline
		\multirow{6}{*}{Commercial} &\multirow{6}{*}{---}	
		&\emph{NumTSteps} &671	  &785	  &891	  &1031	  &1100 \\
		&&\emph{NumNSteps} &1115	  &1244	  &1328	  &1458	  &1515 \\
		&&\emph{Iter}		  &150152	&160254	&161432	&167757	&175452 \\
		&&\emph{AvgIter}    &134.7 	&128.8  &121.6 	&115.1 	&115.8 \\
		&&\emph{Time}     &14.62	&8.45 	&4.71		&2.72		&\textbf{1.54}  \\
		&&\emph{Speedup}  & 1.00 	&1.73 	&3.10 	&5.37 	&9.49 \\
		\hline
		\multirow{24}{*}{Ours} &
		\multirow{6}{*}{0}	
		&\emph{NumTSteps} &164 	  &164 	  &164 	  &164 	  &164  \\
		&&\emph{NumNSteps} &239 	  &239 	  &239    &239 	  &239  \\
		&&\emph{Iter}		  &5823	  &5822	  &5818	  &5822	  &5827 \\
		&&\emph{AvgIter}    &24.4 	&24.4 	&24.3 	&24.4 	&24.4  \\
		&&\emph{Time}     &1.48 	&0.88 	&0.57 	&0.43 	&0.39 \\
		&&\emph{Speedup}  &	1.00 	&1.68 	&2.60 	&3.44 	&3.79 \\
		\cline{2-8}
		&\multirow{6}{*}{20}	
		&\emph{NumTSteps} &164 	  &164 	  &164 	  &164 	  &164  \\
		&&\emph{NumNSteps} &239 	  &239 	  &239    &239 	  &239  \\
		&&\emph{Iter}		  &5855		&5854		&5853		&5855		&5857 \\
		&&\emph{AvgIter}   	&24.5 	&24.5 	&24.5 	&24.5 	&24.5  \\
		&&\emph{Time}    	&1.50 	&0.88 	&0.57 	&0.42 	&0.38 \\
		&&\emph{Speedup} 	&1.00 	&1.70 	&2.63 	&3.57 	&3.95 \\
		\cline{2-8}
		&\multirow{6}{*}{\textbf{30}}	
		&\emph{NumTSteps} &164 	  &164 	  &164 	  &164 	  &164  \\
		&&\emph{NumNSteps} &239 	  &239 	  &239    &239 	  &239  \\
		&&\emph{Iter}		  &6309		&6308		&6315		&6317		&6319 \\
		&&\emph{AvgIter}   	&26.4 	&26.4 	&26.4 	&26.4 	&26.4  \\
		&&\emph{Time}    	&1.52 	&0.88 	&0.56 	&0.40 	&\textbf{0.36} \\
		&&\emph{Speedup} 	&	1.00 	&1.73 	&2.71 	&3.80 	&4.22 \\
		\cline{2-8}
		&\multirow{6}{*}{40}	
		&\emph{NumTSteps} &164		&164		&164		&164		&165 \\
		&&\emph{NumNSteps} &246		&246		&246		&246		&247 \\
		&&\emph{Iter}		  &7033		&7034		&7037		&7041		&7064 \\
		&&\emph{AvgIter}   	&28.6 	&28.6 	&28.6 	&28.6 	&28.6  \\
		&&\emph{Time}    	&1.68 	&0.96 	&0.60 	&0.42 	&0.37 \\
		&&\emph{Speedup} 	&1.00 	&1.75 	&2.80 	&4.00 	&4.54  \\
		\bottomrule
	\end{tabular}
\end{table}

\reftabs{tab:tNavigator-vs-HiSim-OMP-SPE10-2P} and \ref{tab:tNavigator-vs-HiSim-GPU-SPE10-2P} present the number of time steps (\emph{NumTSteps}), the number of Newton iterations (\emph{NumNSteps}), the number of linear iterations (\emph{Iter}), the average number of linear iterations per Newton iteration (\emph{AvgIter}), the total simulation time (\emph{Time}), and the parallel speedup (\emph{Speedup}) for the OpenMP and GPU versions of commercial and our simulators, respectively. For the OpenMP version of the commercial simulator, when the number of threads is changed from 1 to 16, the simulation time is reduced from 14.62 (h) to 1.54 (h). At this point, the maximum speedup reaches 9.49 and the average number of linear iterations per Newton iteration exceeds 115. In our simulator, when $\mu=30$ and the number of threads varied from 1 to 16, the simulation time was reduced from 1.52 (h) to 0.36 (h). At this point, the maximum speedup reaches 4.22 and the average number of linear iterations per Newton iteration is 26.4. This indicates that the proposed method can speed up the simulation time by 4.27 times compared with the commercial simulator. On the one hand, the commercial software chooses some algorithms with high parallel speedup. Usually, such algorithms are easier to parallelize, while they take more iterations to converge. It can be seen from the average number of linear iterations per Newton iteration---Its \emph{AvgIter} is about 4 times more than ours. On the other hand, increasing the number of iterations also improves the parallel speedup. This is because the proportion of solver in the SOLVE increases, and the parallel speedup of the SOLVE phase is usually higher than that of SETUP. Finally, it is worth mentioning that we only parallelize the linear solver in our simulator, while the rest of our simulator is still sequential.

\begin{table}[h]
	\centering
	\setlength{\tabcolsep}{8pt} 
	\renewcommand\arraystretch{1.5} 
	\caption{NumTSteps, NumNSteps, Iter, AvgIter, Time(h), and Speedup (compared with the commercial simulator) comparisons of the GPU version of commercial and our simulators for the two-phase SPE10 problem.}\label{tab:tNavigator-vs-HiSim-GPU-SPE10-2P}	
	\begin{tabular}{c|c|c|c|c|c|c|c}
		\hline
		Simulators	&$\mu$ &\emph{NumTSteps}&\emph{NumNSteps}&\emph{Iter}		&\emph{AvgIter}	&\emph{Time} &\emph{Speedup} \\
		\hline
		Commercial	&---	&1004 &1431 &170276	 &119.0 &3.070 	&---      \\
		\hline
		\multirow{4}{*}{Ours}
		&0	&164 &239 &5525	 &23.1 &0.387 &7.93   \\
		&20	&164& 240 &5659	 &23.6 &0.358 &8.57   \\
		&30	&165 &240 &6182	 &25.8 &0.280 &10.97  \\
		&\textbf{40} &167 & 244 &6740	 &\textbf{27.6}	&\textbf{0.278} &\textbf{11.05}  \\
		\hline
	\end{tabular}
\end{table}

According to~\reftab{tab:tNavigator-vs-HiSim-GPU-SPE10-2P}, we can summarize similar conclusions for GPUs. Compared with the commercial simulator, our obtained numbers of time steps, Newton iterations, and linear iterations are much smaller, and the simulation time is shorter. Especially when $\mu=40$, the speedup of our simulator can achieve 11.05 compared to commercial simulator.

\subsection{Three-phase SPE10}
A three-phase benchmark problem is obtained by changing the fluid properties of the original two-phase SPE10 \cite{2001TenthSPE}. We provide modification information from two-phase to three-phase for the SPE10 benchmark problem. Some keywords and values are added to the input file of the original two-phase SPE10 problem. We change the phase states, fluid properties, and relative permeabilities as follows.
\begin{itemize}
	\item Phase states:
	
	\begin{table}[h]
		\centering	
		\caption{Phase states for two-phase and three-phase.} \label{tab:phase-state}
		\begin{tabular}{cc}
			\toprule
			Two-phase& Three-phase \\
			\midrule
			OIL		&  OIL		\\
			WATER	&  WATER	\\
			&  GAS		\\
			&  DISGAS	\\
			\bottomrule
		\end{tabular}
	\end{table}
	
	\begin{remark}
		\rm{Note that the keyword ``DISGAS'' indicates that dissolved gas in oil is considered in three-phase example.}
	\end{remark}

	\item Fluid properties:
	
	\begin{table}[h]
		\centering	
		\caption{PVDO: PVT properties of dead oil (no dissolved gas) for two-phase. PVCO: PVT properties of live oil in compressibility form (with dissolved gas) for three-phase. Left: PVDO, right: PVCO.}\label{tab:2P-PVDO}
		\begin{tabular}{ccc}
			\toprule
			$ P_o $     & $ B_o $     & $\mu_o$ \\
			\midrule
			300    &  1.05   &    2.85 \\    
			800    &  1.02   &    2.99 \\    
			8000   &  1.01   &    3.00 \\  
			\bottomrule
		\end{tabular}
		\qquad
		\begin{tabular}{cccccc}
			\toprule
			$P_{bub}$	&$ R_{so} $ &$B_o$  &$ \mu_{o} $ & $\kappa_o$ & $vc_o $ \\
			\midrule
			400     &0.0165    &1.01200     &3.5057  &1.1388e-6   &0 \\
			4000    &0.1130    &1.01278     &2.9972  &1.1388e-6   &0 \\
			10000   &0.1810    &1.15500     &2.6675  &1.1388e-6   &0 \\
			\bottomrule
		\end{tabular}
	\end{table}
	
	where $P_o$, $B_o$, and $ \mu_{o} $ denote the pressure, volume coefficient, and viscosity coefficient of oil phase, respectively. $P_{bub}$ and $ R_{so} $ denote the bubble point pressure for oil and dissolved gas-oil ratio; $B_o$ and $ \mu_{o} $ denote the volume coefficient and viscosity coefficient of saturated oil; $\kappa_o$ and $vc_o $ denote the compressibility and viscosity compressibility of undersaturated oil.
	
	\begin{table}[h]
		\centering	
		\caption{PMAX: Maximum pressure during the simulation.} \label{tab:maximum-pressure}
		\begin{tabular}{cccc}
			\toprule
			$ P_{\max} $& $ \hat{P}_{\max} $ & $ \hat{P}_{\min} $ & $ N_{nodes} $  \\
			\midrule
			16000 &0.0 &1.0E+20 &30 \\
			\bottomrule
		\end{tabular}		
	\end{table}
	where $ P_{\max} $ denotes the maximum pressure that could be reached during the simulation. $ \hat{P}_{\max} $, $ \hat{P}_{\min} $, and $ N_{nodes} $ denote maximum pressure to extend the range of pressures, minimum pressure to extend the lower range of pressures, and the number of nodes for checking total compressibility of the system.
	
	\begin{table}[h]
		\centering	
		\caption{PVDG: PVT properties of dry gas (no vaporized oil).
			left: two-phase, right: three-phase.}\label{tab:2P-PVDG}
		\begin{tabular}{ccc}
			\toprule
			$P_g$	&$B_g$  &$ \mu_{g} $ \\
			\midrule
			300     & 1.98  &0.0162 \\
			800     & 1.11  &0.0197 \\
			8000    & 0.60  &0.0330 \\
			&       &       \\
			\bottomrule
		\end{tabular}
		\qquad
		\begin{tabular}{ccc}
			\toprule
			$P_g$	&$B_g$  &$ \mu_{g} $ \\
			\midrule
			400    &  1.96 & 0.0140  \\
			4000   &  0.84 & 0.0160  \\
			8000   &  0.59 & 0.0175 \\
			10000  &  0.42 & 0.0195 \\
			\bottomrule
		\end{tabular}
	\end{table}
	
	where $P_g$, $B_g$, and $ \mu_{g} $ denote the pressure, volume coefficient, and viscosity coefficient of gas phase, respectively.

	\item Gas-oil relative permeabilities and capillary pressure:
	
	\begin{table}[h]
		\centering	
		\caption{SGOF: gas/oil relative permeabilities and capillary pressure versus gas saturation for three-phase.}\label{tab:SGOF}
		\begin{tabular}{cccc}
			%			\hline
			\toprule
			$S_g$	&$\kappa_{rg}$  &$ \kappa_{rog} $  & $ p_{cog} $  \\
			%			\hline
			\midrule
			0.00   &  0.0000     &  1.00    &    0.0  \\
			0.04   &  0.0000     &  0.60    &    0.2  \\
			0.10   &  0.0220     &  0.33    &    0.5  \\
			0.20   &  0.1000     &  0.10    &    1.0  \\
			0.30   &  0.2400     &  0.02    &    1.5  \\
			0.40   &  0.3400     &  0.00    &    2.0  \\
			0.50   &  0.4200     &  0.00    &    2.5  \\
			0.60   &  0.5000     &  0.00    &    3.0  \\
			0.70   &  0.8125     &  0.00    &    3.5  \\
			0.80   &  1.0000     &  0.00    &    3.9  \\    
			%			\hline
			\bottomrule
		\end{tabular}
	\end{table}
	
	where $S_g$, $\kappa_{rg}$, $ \kappa_{rog} $, and $ p_{cog} $ denote the gas saturation, gas relative permeability, oil relative permeability (when oil, gas and connate water are present), and capillary pressure between the gas and oil phases, respectively.
\end{itemize}

Next we are going to verify convergence behavior and parallel performance of the PGS-SCM method for the three-phase SPE10 problem. Also, the correctness and parallel performance for the ASCPR-GMRES-OMP and ASCPR-GMRES-CUDA methods are tested and the performance of our and commercial simulators are compared.

\subsubsection{PGS-SCM method}
\begin{example}\label{ex:1}
	\rm{The three-phase example is considered, and the numerical simulation conducted for 100 days. We use the different number of threads ($ \mathrm{NT}=1,2,4,8$, and $16$) to test the parallel performance of CPR-GMRES-PGS-NO and CPR-GMRES-PGS-SCM. The impacts of the different strength thresholds $ \theta $ ($\theta=0,0.05,0.1$, and $0.3 $) on CPR-GMRES-PGS-SCM are also tested.}
\end{example}

\reftab{tab:MC-GS-3P} lists the experimental results of Example \ref{ex:1}, including the total number of linear iterations (\emph{Iter}), the total solution time (\emph{Time}), and the parallel speedup (\emph{Speedup}).
\begin{table}[H]
	\centering
	\setlength{\tabcolsep}{4.5pt} 
	\renewcommand\arraystretch{1.5} 
	\caption{Iter, Time(s), and Speedup of the two solvers with different NT for the three-phase SPE10 problem.}
	\label{tab:MC-GS-3P}
	\begin{tabular}{c|c|c|ccccc}
		\toprule%[1.2pt]
		Solvers & $\theta$ & NT  & 1& 2 & 4 & 8 & 16 \\
		\hline
		\multirow{3}{*}{CPR-GMRES-PGS-NO}
		&\multirow{3}{*}{---}	&\emph{Iter}	&3336			&3338			&3372			&3439 		&3527		\\
		&											&\emph{Time}		&3406.71 	&2050.70 	&1258.79 	&804.38 	&687.76 \\
		&											&\emph{Speedup} &1.00				&1.66   	&2.71 	  &4.24   	&4.95 	\\		
		\hline
		\multirow{12}{*}{CPR-GMRES-PGS-SCM}
		&\multirow{3}{*}{0}		&\emph{Iter}	&\textbf{3334}	&\textbf{3334}	&\textbf{3334}	&\textbf{3334}	&\textbf{3334} \\
		&											&\emph{Time}		&3489.93 	&1937.48 	&1125.65 	&723.21 	&608.78 \\
		&											&\emph{Speedup} &1.00 	&1.80 	&3.10 	&4.83 	&\textbf{5.73}  \\
		\cline{2-8}	
		&\multirow{3}{*}{0.05}&\emph{Iter}	&3236	&3236	&3237	&3237	&\textbf{3237} \\
		&											&\emph{Time}		&3446.89 	&1920.54 	&1102.48 	&711.16 	&\textbf{598.86} \\
		&											&\emph{Speedup} &1.00 	&1.79 	&3.13 	&4.85 	&\textbf{5.76} \\	
		\cline{2-8}	
		&\multirow{3}{*}{0.1} &\emph{Iter}	&3379	&3381	&3381	&3385	&3385 \\
		&											&\emph{Time}		&3551.89 	&1976.98 	&1141.65 	&745.85 	&624.98 \\
		&											&\emph{Speedup} &1.00 	&1.80 	&3.11 	&4.76 	&5.68 \\	
		\cline{2-8}	
		&\multirow{3}{*}{0.3}	&\emph{Iter}	&3314	&3315	&3470	&3496	&3477 \\
		&											&\emph{Time}		&3421.46	&1909.81	&1154.44	&756.08	  &631.75 \\
		&											&\emph{Speedup} &1.00 	&1.79	&2.96 	&4.53 	&5.42 \\				
		\bottomrule
	\end{tabular}
\end{table}

It can be seen from \reftab{tab:MC-GS-3P} that, for CPR-GMRES-PGS-NO solver, the total number of linear iterations gradually increases as NT increases. When $\mathrm{NT}=16$, the total number of linear iterations increases by 191 compared with the single-thread result, and the speedup is 4.95. This implies that the number of iterations of the parallel GS algorithm based on natural ordering is not stable. That is, as the number of threads increases, the number of iterations increases, which affects the parallel speedup of the solver. 
However, for CPR-GMRES-PGS-SCM solver (when $ \theta=0 $), the total number of linear iterations is not changed with a larger NT. Such results indicate that the proposed PGS-SCM yields the same convergence behavior as the corresponding single-thread algorithm, which verifies the effectiveness of the algorithm.  
Moreover, from the perspective of parallel performance, CPR-GMRES-PGS-SCM gets a higher speedup compared with CPR-GMRES-PGS-NO. For example, when $\mathrm{NT}=16$, the speedup of CPR-GMRES-PGS-SCM increases to 5.73 (from 4.95 to 5.73). Hence, the proposed PGS-SCM method can obtain a better parallel speedup.

Next, we discuss the influence of different strength thresholds $ \theta $ on the parallel performance for CPR-GMRES-PGS-SCM solver.
When $ \theta $ is small (for example, $ \theta=0$ or $0.05$), the total number of linear iterations changes little with the increase of NT and can even be considered unchanged. However, when $\theta$ is larger (for example, $\theta=0.1$ or $0.3 $), the varied range of the total number of linear iterations is enlarged with the increase of NT. These results show that as $ \theta $ increases, the independence of the degrees of freedom is decreased, which affects the number of iterations. If $\mathrm{NT}=16$ is considered, the proposed method obtains some meaningful results. As $ \theta $ increases, the total number of linear iterations  first decreases and then increases, and the speedup first increases and then decreases. Especially when $ \theta = 0.05$, the minimum total number of linear iterations is 3237, and the speedup is the highest, reaching 5.76. All in all, the strength threshold $ \theta $ affects the stability of the number of iterations and parallel performance. In this experiment, when $ \theta=0.05 $, the obtained number of iterations is the least, and the parallel speedup is the highest.

\subsubsection{ASCPR-GMRES-OMP method}
In order to verify the correctness and parallel performance of ASCPR-GMRES-OMP, we discuss the influences of different thresholds $\mu $ on the experimental results.

\begin{example}\label{ex:2}
	\rm{For the three-phase example, the numerical simulation is conducted for 100 days, and $ \theta=0.05 $. We explore the parallel performance of ASCPR-GMRES-OMP for five groups of $ \mu $ (i.e., $ \mu = 0, 10, 20, 30$, and $40$), when $\mathrm{NT} = 1, 2, 4, 8$, and $16$, respectively.}
\end{example}

First, to verify the correctness of ASCPR-GMRES-OMP, we present the field oil production rate, field gas production rate, and average pressure graphs of five groups $\mu $ (when $\mathrm{NT}=16$); see \reffig{fig:spe10-3p-Oil-Pre-omp}.

\begin{figure}[H]
	\subfigure[Field oil production rate]{\label{fig:spe10-3p-Oil-Rate-nt16}
		\begin{minipage}[t]{0.33\linewidth}
			\centering
			\includegraphics[width=5.5cm]{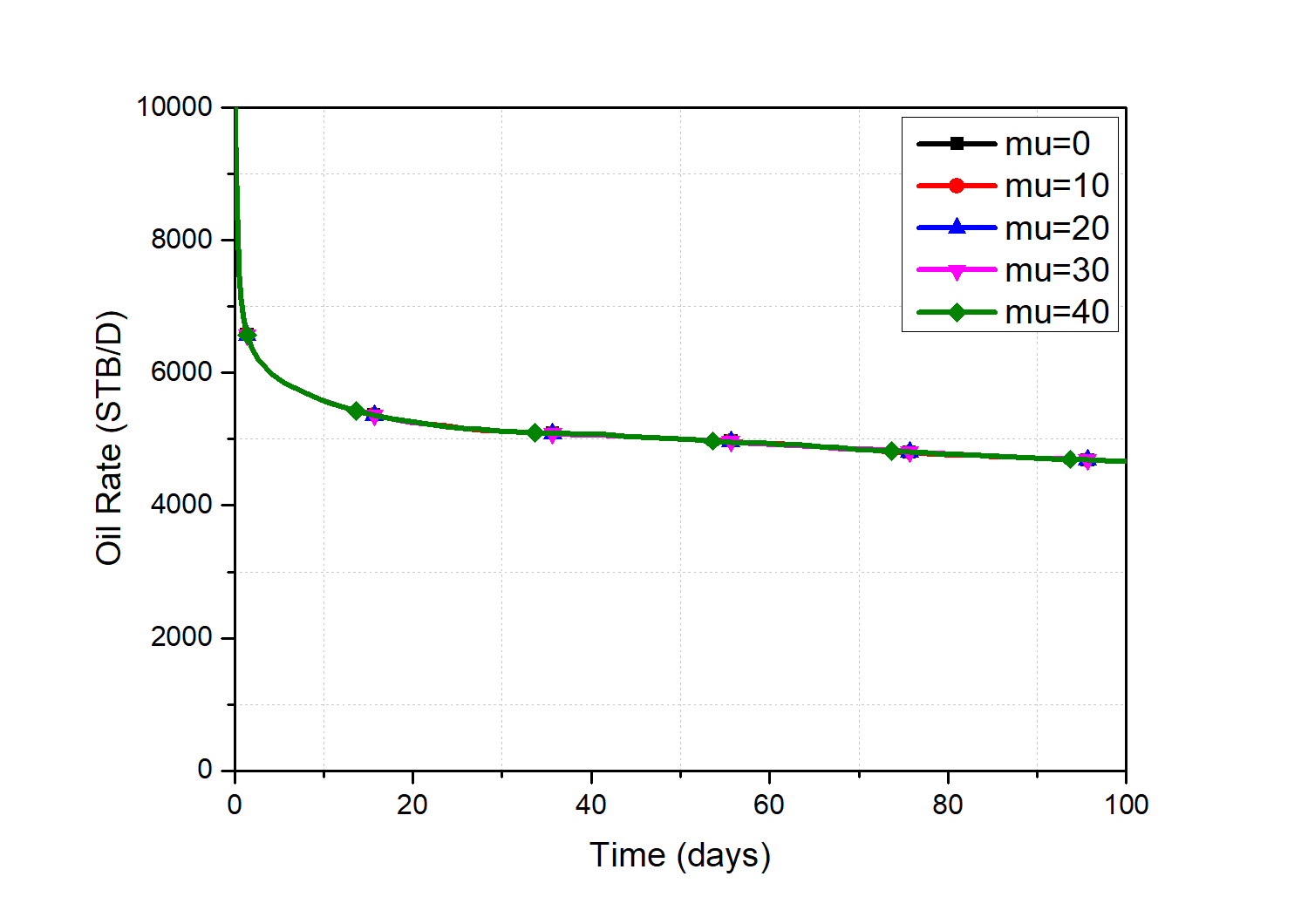}		
		\end{minipage}
	}%	
	\subfigure[Field gas production rate]{\label{fig:spe10-3p-Gas-Rate-nt16}
		\begin{minipage}[t]{0.33\linewidth}
			\centering
			\includegraphics[width=5.5cm]{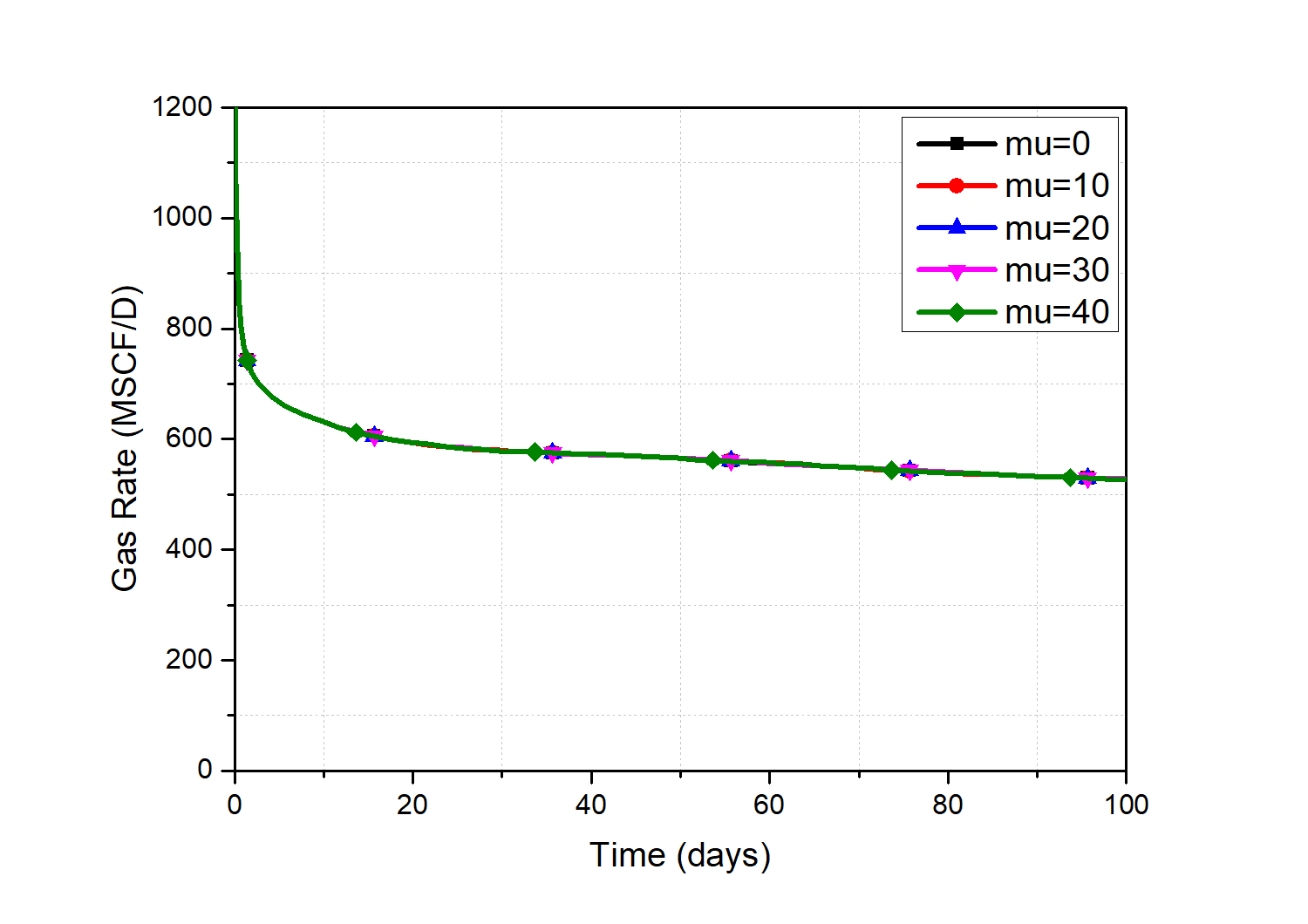}		
		\end{minipage}
	}%	
	\subfigure[Average pressure]{\label{fig:spe10-3p-Ave-Pre-nt16}
		\begin{minipage}[t]{0.33\linewidth}
			\centering
			\includegraphics[width=5.5cm]{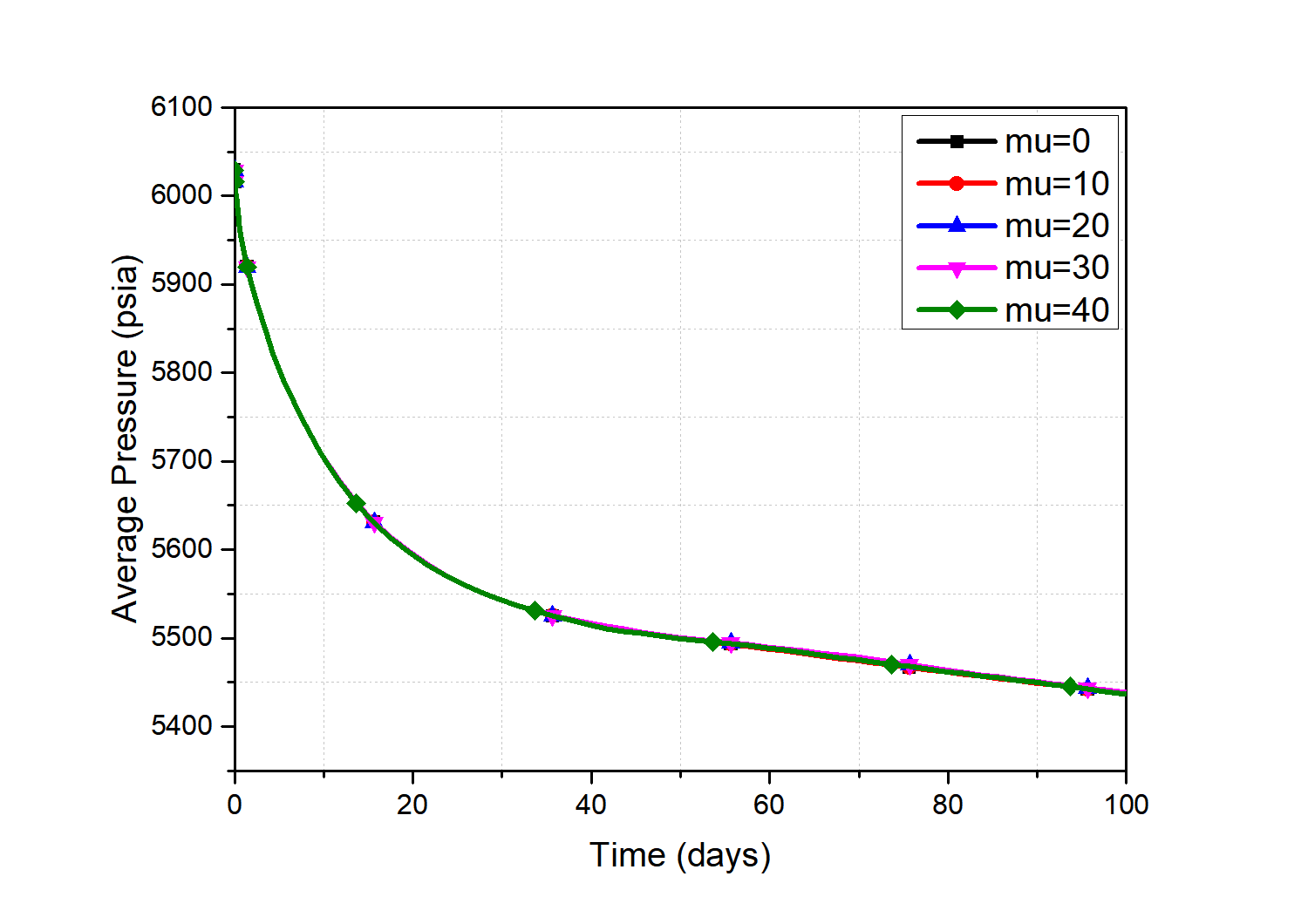}	
		\end{minipage}
	}
	\caption{Comparison charts of field oil production rate, field gas production rate, and average pressure of five groups $ \mu $ for the three-phase SPE10 problem (OpenMP, NT=16).}\label{fig:spe10-3p-Oil-Pre-omp}
\end{figure}

From \reffig{fig:spe10-3p-Oil-Pre-omp}, we can see that the field oil production rate, field gas production rate, and average pressure obtained by the adaptive SETUP CPR preconditioner ($ \mu=10,20,30$, and $40$) and the general CPR preconditioner ($ \mu=0 $) completely coincide, indicating that the ASCPR-GMRES-OMP method is correct.

Then, we also discuss the impacts of different thresholds $\mu $ on the parallel speedup of ASCPR-GMRES-OMP. \reftab{tab:Time-Newton-Linear-SPE10-3P} lists the number of SETUP calls (\emph{SetupCalls}), the ratio of SETUP in the total solution time (\emph{SetupRatio}), the total number of linear iterations (\emph{Iter}), total solution time (\emph{Time}), new parallel speedup ($\text{\emph{Speedup}}^{*}$), and parallel speedup (\emph{Speedup}). 
\begin{table}[H]
	\centering
	\setlength{\tabcolsep}{12pt} 
	\renewcommand\arraystretch{1.2} 
	\caption{SetupCalls, SetupRatio, Iter, Time(s), $\text{Speedup}^{*}$, and Speedup of different $\mu$ and NT for the three-phase SPE10 problem.}
	\label{tab:Time-Newton-Linear-SPE10-3P}	
	\begin{tabular}{cccccccc}
		\toprule
		& $\mu$  & 1& 2 & 4 & 8 & 16\\
		\hline
		\multirow{5}{*}{\emph{SetupCalls}}  
		&0	&178	&178	&178	&178	&178	\\
		&10	&141	&141	&141	&141	&141	\\
		&20	&98	&98	&98	&98	&98	\\
		&30	&25	&25	&25	&25	&25	\\
		&40	&13	&13	&13	&13	&13	\\
		\hline
		\multirow{5}{*}{\emph{SetupRatio}}  
		&0 	&13.13\%	&16.51\%	&22.06\%	&30.68\%	&40.44\% \\
		&10	&12.04\%	&15.01\%	&19.88\%	&28.25\%	&38.75\% \\
		&20	&8.76\%		&10.73\%	&14.05\%	&20.74\%	&33.23\% \\
		&30	&5.33\%		& 6.23\%	& 8.11\%	&11.39\%	&25.06\% \\
		&40	&4.67\%		& 5.47\%	& 6.91\%	&10.53\%	&19.83\% \\
		\hline
		\multirow{5}{*}{\emph{Iter}}  
		&0	&3236&	3236&		3237&		3237&		3237\\
		&10	&3253&	3253&		3253&		3253&		3253\\
		&20	&3464&	3463& 	3460& 	3461& 	3460\\ 
		&30	&3891& 	3891& 	3890& 	3889& 	3996\\ 
		&40	&4111& 	4111& 	4118& 	4125& 	4106\\ 
		\hline
		\multirow{5}{*}{\emph{Time}}
		&0	&3446.89 	&1920.54 	&1102.48 	&711.16 	&598.86  \\
		&10	&3348.65 	&1853.37 	&1080.48 	&685.87 	&573.68  \\
		&\textbf{20}	&\textbf{3446.32} 	&\textbf{1885.40} 	&\textbf{1063.75} 	&\textbf{654.28} 	&\textbf{522.75}  \\
		&30	&3973.86 	&2145.97 	&1165.28 	&676.72 	&558.70  \\
		&40	&4229.11 	&2239.03 	&1210.80 	&717.30 	&582.31  \\
		\hline
		\multirow{5}{*}{$\text{\emph{Speedup}}^{*}$}  
		&0	&1.00 	&1.79 	&3.13 	&4.85 	&5.76  \\
		&10	&1.03 	&1.86 	&3.19 	&5.03 	&6.01  \\
		&\textbf{20}	&\textbf{1.00} 	&\textbf{1.83} 	&\textbf{3.24} 	&\textbf{5.27} 	&\textbf{6.59}  \\
		&30	&0.87 	&1.61 	&2.96 	&5.09 	&6.17  \\
		&40	&0.82 	&1.54 	&2.85 	&4.81 	&5.92  \\
		\hline
		\multirow{5}{*}{\emph{Speedup}}  
		&0	&1.00 	&1.79 	&3.13 	&4.85 	&5.76  \\
		&10	&1.00 	&1.81 	&3.10 	&4.88 	&5.84  \\
		&20	&1.00 	&1.83 	&3.24 	&5.27 	&6.59  \\
		&30	&1.00 	&1.85 	&3.41 	&5.87 	&7.11  \\
		&40	&1.00 	&1.89 	&3.49 	&5.90 	&7.26  \\
		\bottomrule
	\end{tabular}
\end{table}

According to the results of \reftab{tab:Time-Newton-Linear-SPE10-3P}, we first discuss the number of SETUP calls, the ratio of SETUP in the total solution time, and the total number of linear iterations. For the number of SETUP calls and the ratio of SETUP in the total solution time, when NT is fixed, both the number of SETUP calls and the ratio of SETUP in the total solution time decrease as $\mu$ increases. Especially when $\mu=40$, there are only 13 SETUP calls. When $\mu$ is fixed, the number of SETUP calls is not changed with the increase of NT, but the ratio of SETUP in the total solution time is gradually increased (this also reflects fact that the parallel speedup of SOLVE is higher than that of SETUP). 
For the total number of linear iterations, when NT is fixed, the total number of linear iterations is gradually increasing as $\mu$ increases. In short, these results indicate that the number of SETUP calls is significantly decreased with the increase of $\mu$, but the total number of linear iterations is also increased.

What is more, we discuss the total solution time and the new parallel speedup (only discuss the impacts of the change of $\mu$ on the results). Let's take NT=16 as an example. When $\mu$ increases by 0, 10, 20, 30, and 40 in turn, the total solution time is first decreased and then increased, and the new parallel speedup  is first increased and then decreased. Especially when $\mu=20$, the total solution time is 522.75 (s), and the corresponding the new parallel speedup is 6.59. Compared with the case of the general CPR preconditioner, the total solution time of ASCPR-GMRES-OMP is reduced from 598.86 (s) to 522.75 (s), and the new parallel speedup is increased from 5.76 to 6.59. These results show that the proposed method can further improve the parallel performance of the solver.

Finally, we discuss the changes in the parallel speedup. As $\mu$ increases, the parallel speedup is increased. Specifically, when $\mu=40$ and $ \mathrm{NT}=16 $, the parallel speedup reaches 7.26. These results are consistent with what we expected. However, our goal is not to pursue the highest parallel speedup but the highest new parallel speedup (or the shortest total solution time).

\subsubsection{ASCPR-GMRES-CUDA method}
In order to verify the correctness and parallel performance of ASCPR-GMRES-CUDA, we explore the influences of different thresholds $\mu $ on the experimental results.

\begin{example}\label{ex:4}
	\rm{For the three-phase example, the numerical simulation is conducted for 100 days, and $ \theta=0.05 $. We explore the impacts of $ \mu $ ($ \mu = 0, 10, 20, 30$, and $40$) on the parallel performance of ASCPR-GMRES-CUDA.}
\end{example}

To verify the correctness of ASCPR-GMRES-CUDA, we plot the field oil production rate, field gas production rate, and average pressure graphs of five sets $\mu$, as shown in \reffig{fig:spe10-3p-Oil-Pre-gpu}.
\begin{figure}[H]
	\subfigure[Field oil production rate]{\label{fig:spe10-3p-Oil-Rate-gpu}
		\begin{minipage}[t]{0.33\linewidth}
			\centering
			\includegraphics[width=5.5cm]{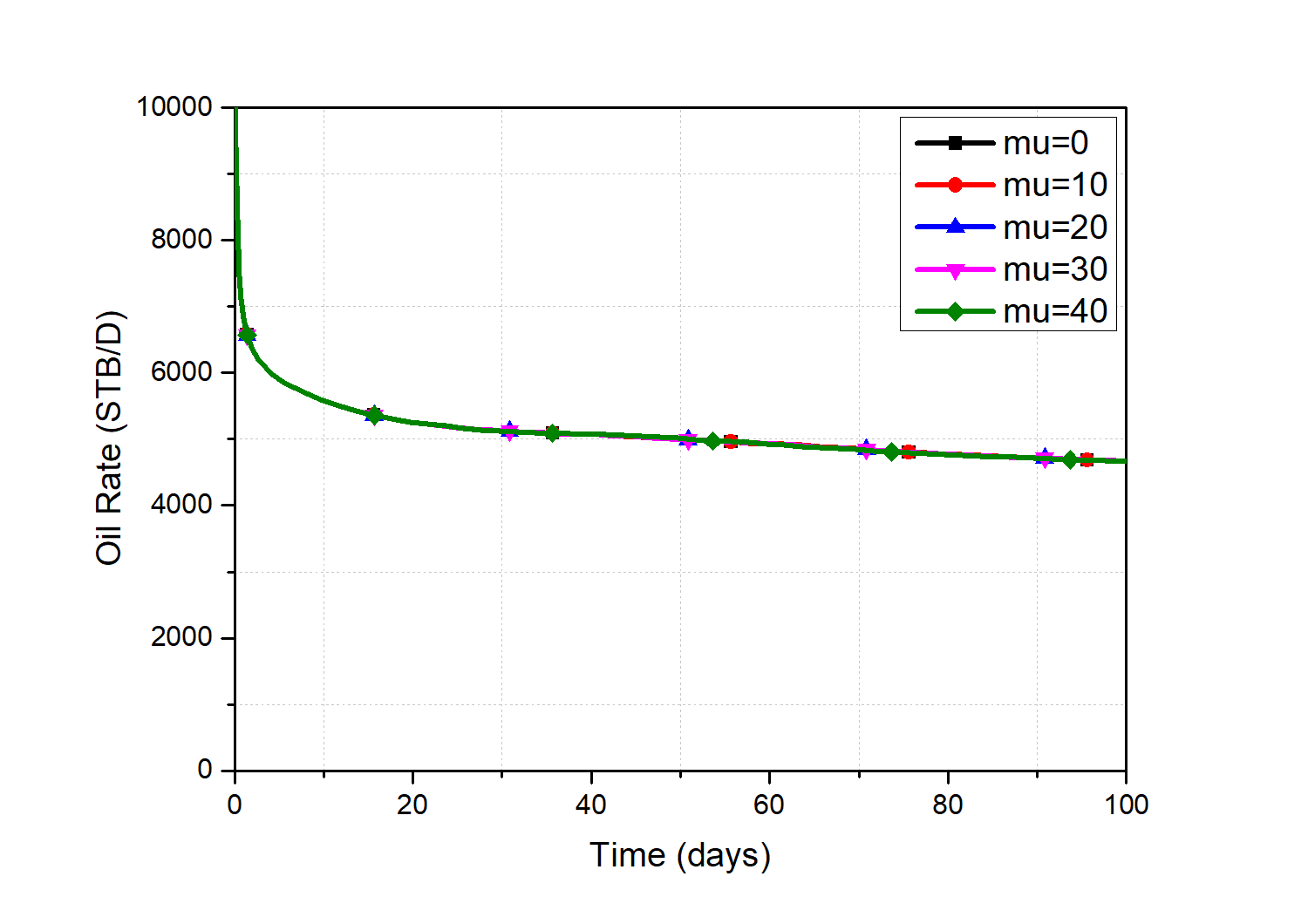}		
		\end{minipage}
	}%	
	\subfigure[Field gas production rate]{\label{fig:spe10-3p-Gas-Rate-gpu}
		\begin{minipage}[t]{0.33\linewidth}
			\centering
			\includegraphics[width=5.5cm]{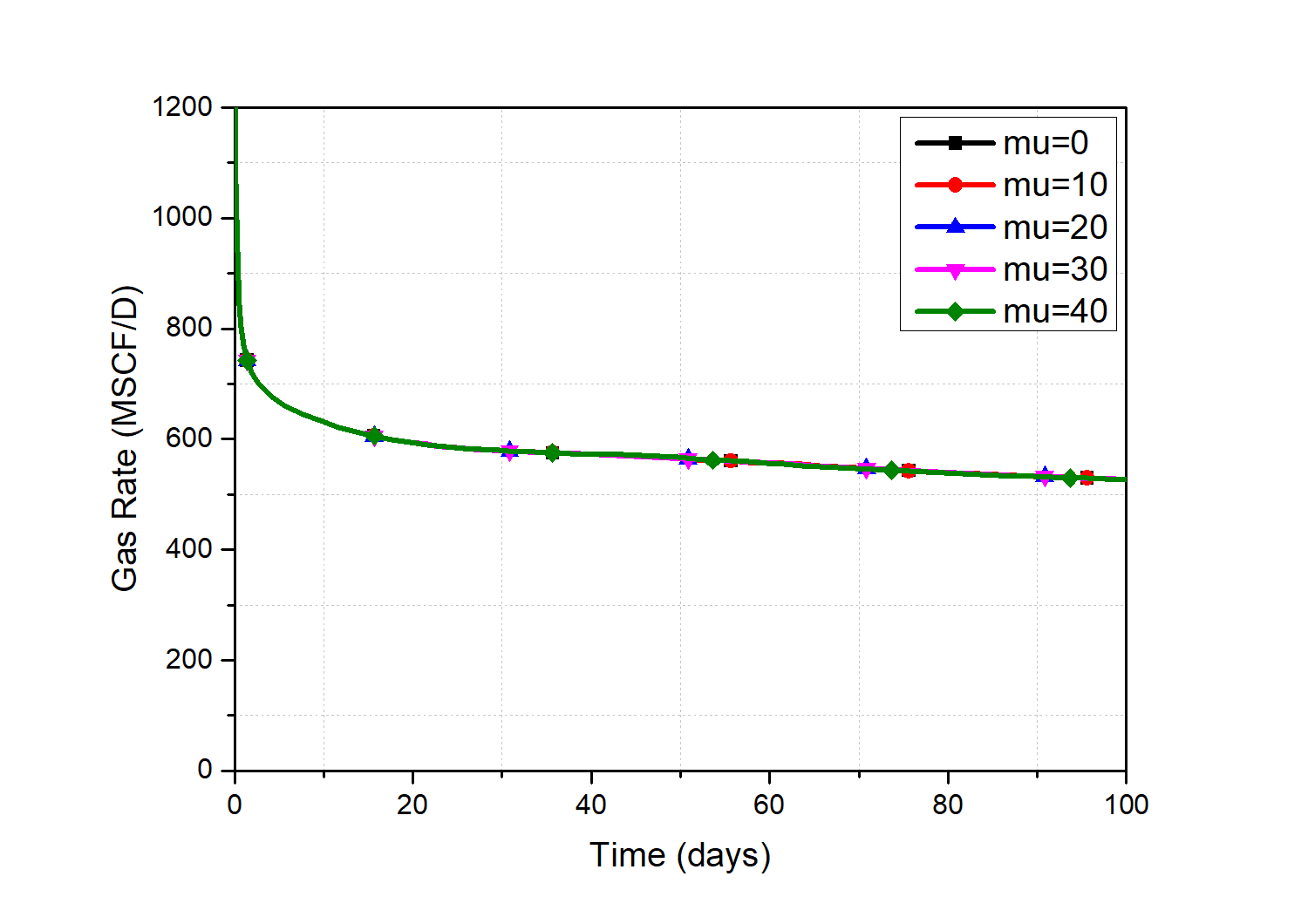}		
		\end{minipage}
	}%	
	\subfigure[Average pressure]{\label{fig:spe10-3p-Ave-Pre-gpu}
		\begin{minipage}[t]{0.33\linewidth}
			\centering
			\includegraphics[width=5.5cm]{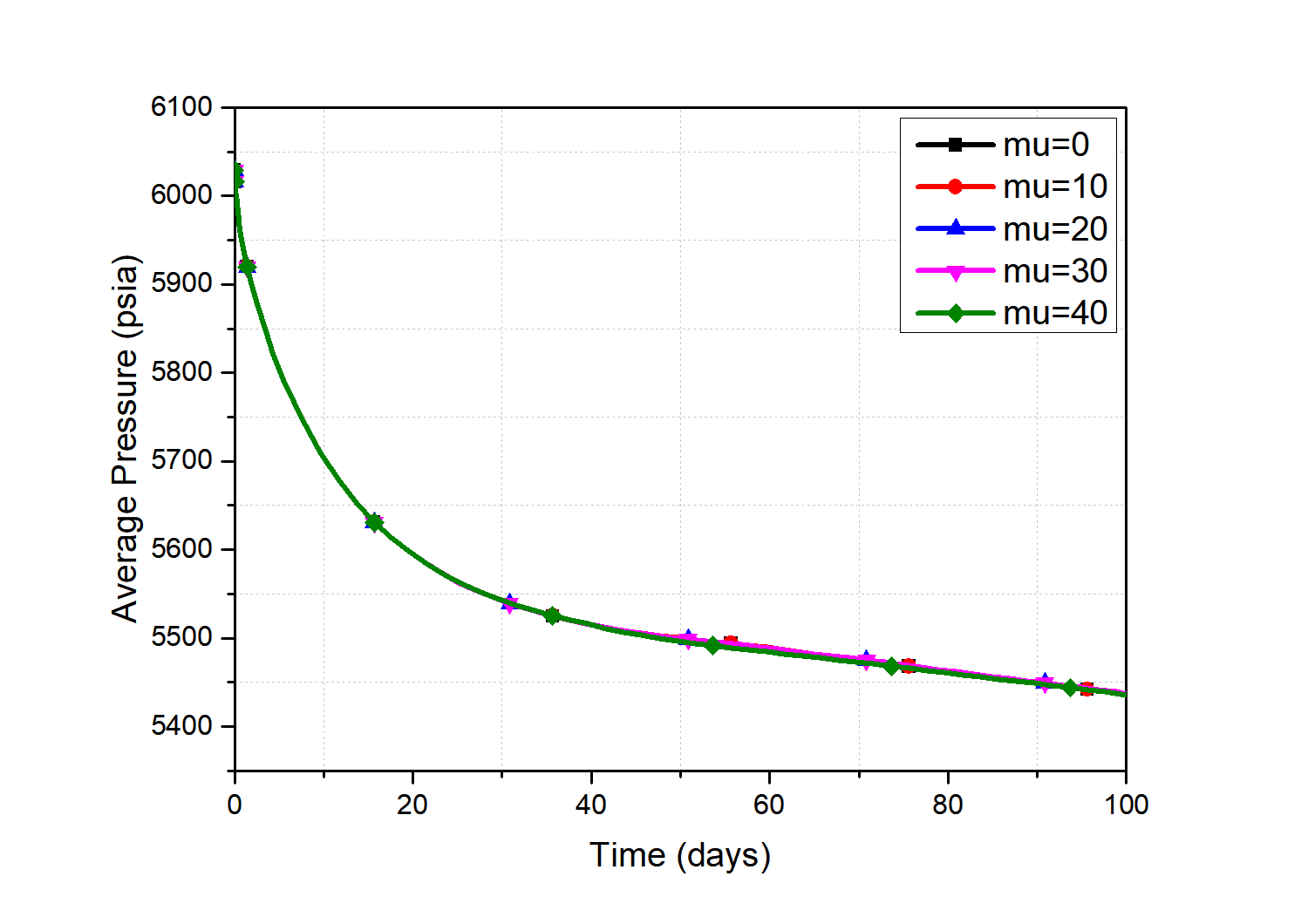}	
		\end{minipage}
	}
	\caption{The field oil production rate, field gas production rate, and average pressure comparison chart of five sets $\mu$ for the three-phase SPE10 problem (CUDA).}
	\label{fig:spe10-3p-Oil-Pre-gpu}
\end{figure}

From \reffig{fig:spe10-3p-Oil-Pre-gpu}, we can see that the field oil production rate, field gas production rate, and average pressure obtained by the adaptive SETUP CPR preconditioner (i.e., $ \mu=10, 20, 30$, and $40$) and the general CPR preconditioner (i.e., $ \mu=0 $) completely coincide, indicating that the ASCPR-GMRES-CUDA is correct.

In addition, we study the parallel performance of the ASCPR-GMRES-CUDA. \reftab{tab:GPU-results-3p} lists \emph{SetupCalls}, \emph{SetupRatio}, \emph{Iter}, \emph{Time}, and \emph{Speedup} of ASCPR-GMRES-CUDA. The single-thread calculation results of the OpenMP version program (denoted as ASCPR-GMRES-OMP(1)) are also added for comparison.
\begin{table}[H]
	\centering
	\setlength{\tabcolsep}{7.0pt} 
	\renewcommand\arraystretch{1.5} 
	\caption{SetupCalls, SetupRatio, Iter, Time(s), and Speedup (compared to ASCPR-GMRES-OMP(1)) of the different $\mu$ for the three-phase SPE10 problem.}
	\label{tab:GPU-results-3p}
	\begin{tabular}{c|c|c|c|c|c|c}
		\hline
		Solvers & $\mu$& \emph{SetupCalls}&	\emph{SetupRatio} & \emph{Iter}&	\emph{Time}&	\emph{Speedup} \\
		\hline
		ASCPR-GMRES-OMP(1) &0	  &178	& 13.13\% &3236	&3446.89 	&---  \\
		\hline
		\multirow{5}{*}{ASCPR-GMRES-CUDA}
		&0		&178	&74.01\%	&3270	&594.96 	& 5.79 \\
		&10		&146	&70.97\%	&3302	&534.33 	& 6.45 \\
		&20		&69		&61.20\%	&3424	&413.12 	& 8.34 \\
		&\textbf{30}		&\textbf{24}		&\textbf{46.24\%}	&\textbf{4065}	&\textbf{355.80} 	& \textbf{9.69} \\
		&40		&17		&44.21\%	&4239	&362.97  	& 9.50 \\
		\hline
	\end{tabular}
\end{table}

It can be seen from \reftab{tab:GPU-results-3p} that for ASCPR-GMRES-CUDA solver, as $ \mu $ increases, both the number of SETUP calls and the ratio of SETUP in the total solution time decreases, the total number of linear iterations gradually increases, and the parallel speedup first increases and then decreases. The CUDA version can be 5.79 times faster than the ASCPR-GMRES-OMP(1) (when $ \mu=0 $). In particular, when $ \mu=30 $, compared with $ \mu=0 $, the total solution time of ASCPR-GMRES-CUDA is reduced from 594.96 (s) to 355.80 (s), which can be 1.67 times faster. At the same time, compared with ASCPR-GMRES-OMP(1), the speedup of ASCPR-GMRES-CUDA reaches 9.69. Therefore, the proposed method obtains distinct acceleration effects and is more suitable for GPU architecture.

Finally, we present the effects of different $\mu$ on the timestep size $ \Delta t $ for ASCPR-GMRES-OMP and ASCPR-GMRES-CUDA solvers, see~\reffig{fig:DeltaT-SPE10-3P}. It can be seen from~\reffig{fig:DeltaT-SPE10-3P-OMP-MU} that when $\mu=0$, the timestep size is consistent for ASCPR-GMRES-OMP with different thread numbers. From~\reffigs{fig:DeltaT-SPE10-3P-OMP-NT} and~\ref{fig:DeltaT-SPE10-3P-GPU}, 
the change in the timestep sizes of the ASCPR-GMRES-OMP and ASCPR-GMRES-CUDA solvers can be ignored when $\mu$ changes.

\begin{figure}[H]
	\centering
	\subfigure[OpenMP, $ \mu=0 $]{\label{fig:DeltaT-SPE10-3P-OMP-MU}
		\begin{minipage}[t]{0.33\linewidth}
			\centering
			\includegraphics[width=5.5cm]{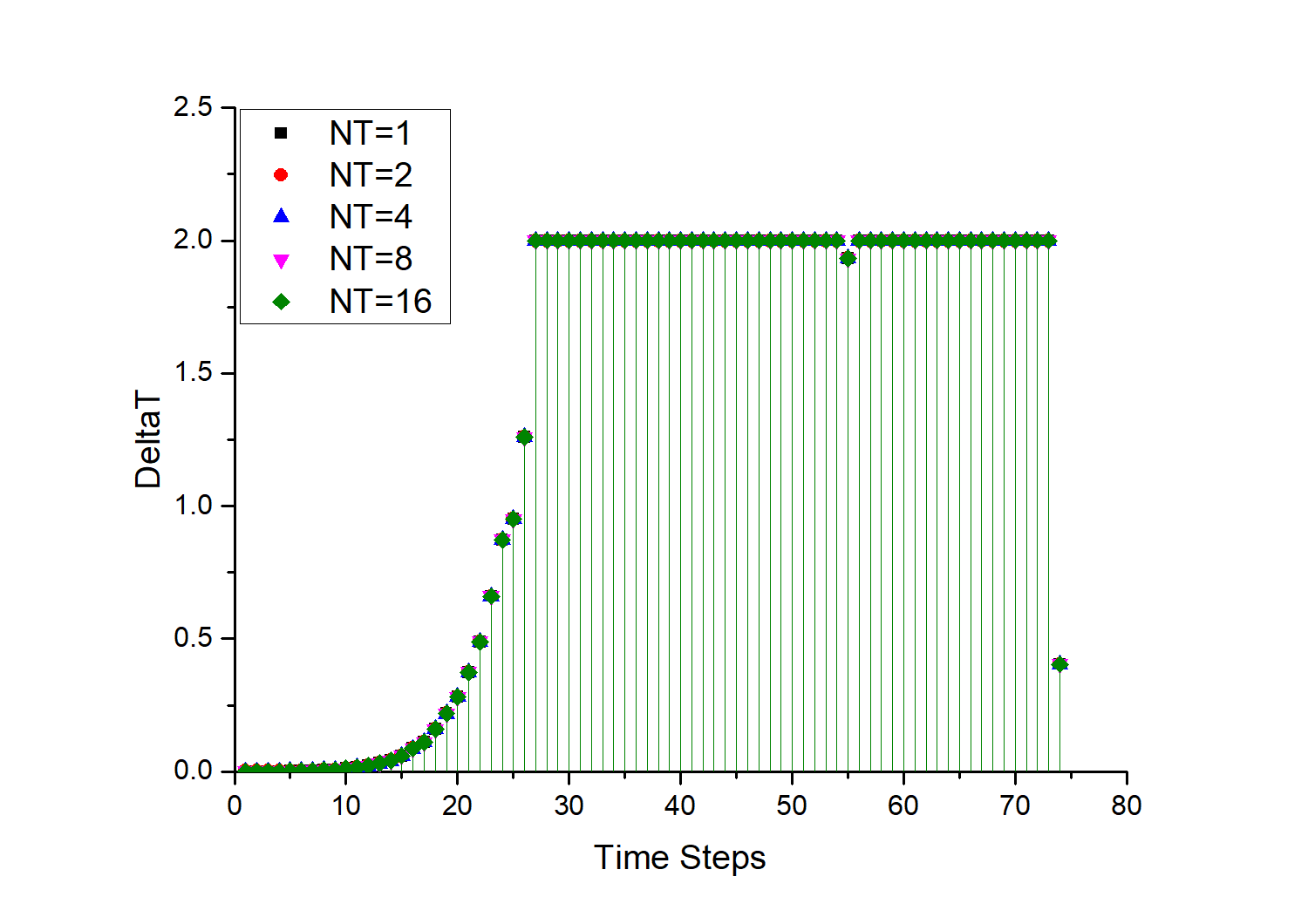}
		\end{minipage}
	}%	
	\subfigure[OpenMP, $ NT=16 $]{\label{fig:DeltaT-SPE10-3P-OMP-NT}
		\begin{minipage}[t]{0.33\linewidth}
			\centering
			\includegraphics[width=5.5cm]{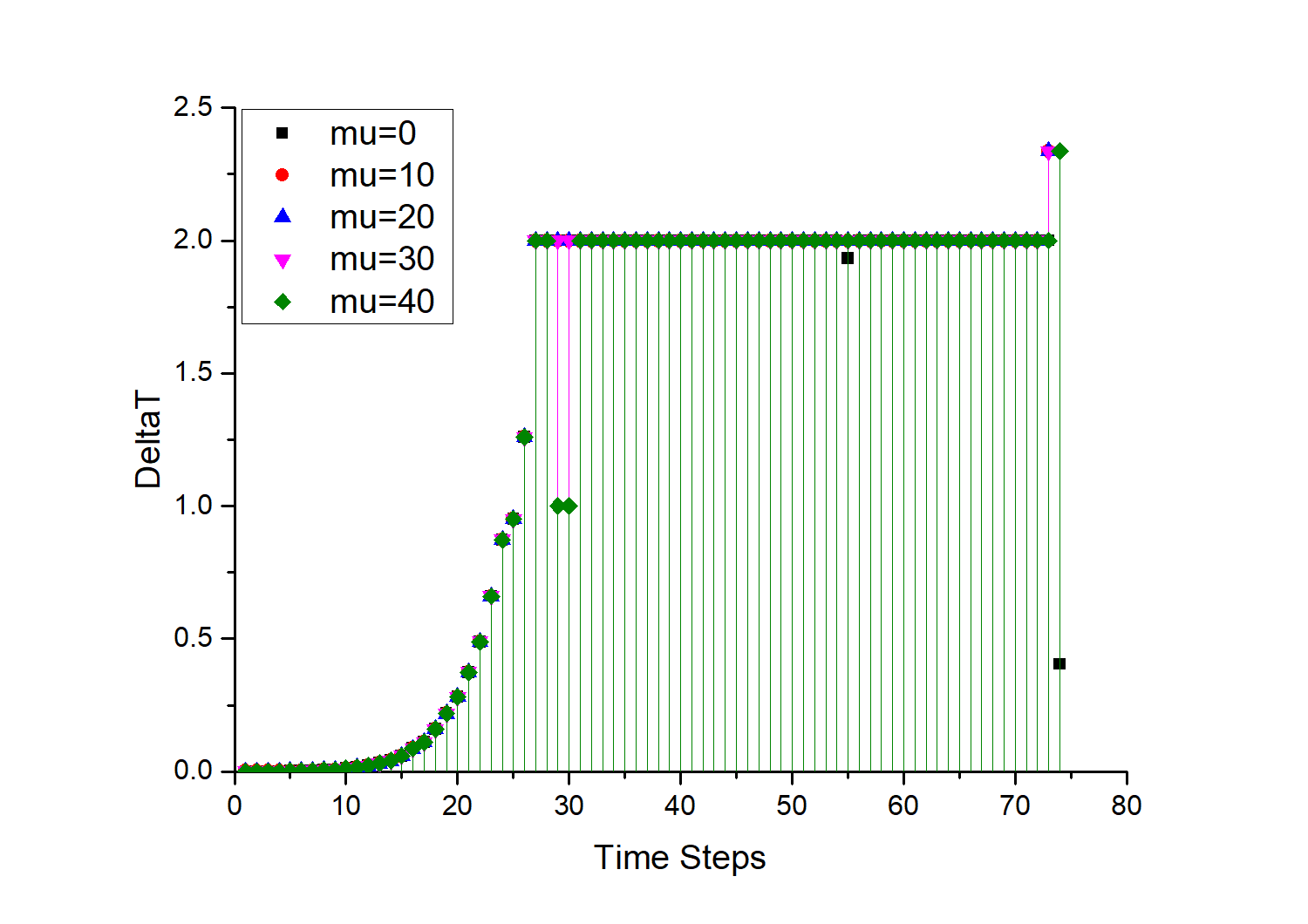}
		\end{minipage}
	}%
	\subfigure[GPU]{\label{fig:DeltaT-SPE10-3P-GPU}
		\begin{minipage}[t]{0.33\linewidth}
			\centering
			\includegraphics[width=5.5cm]{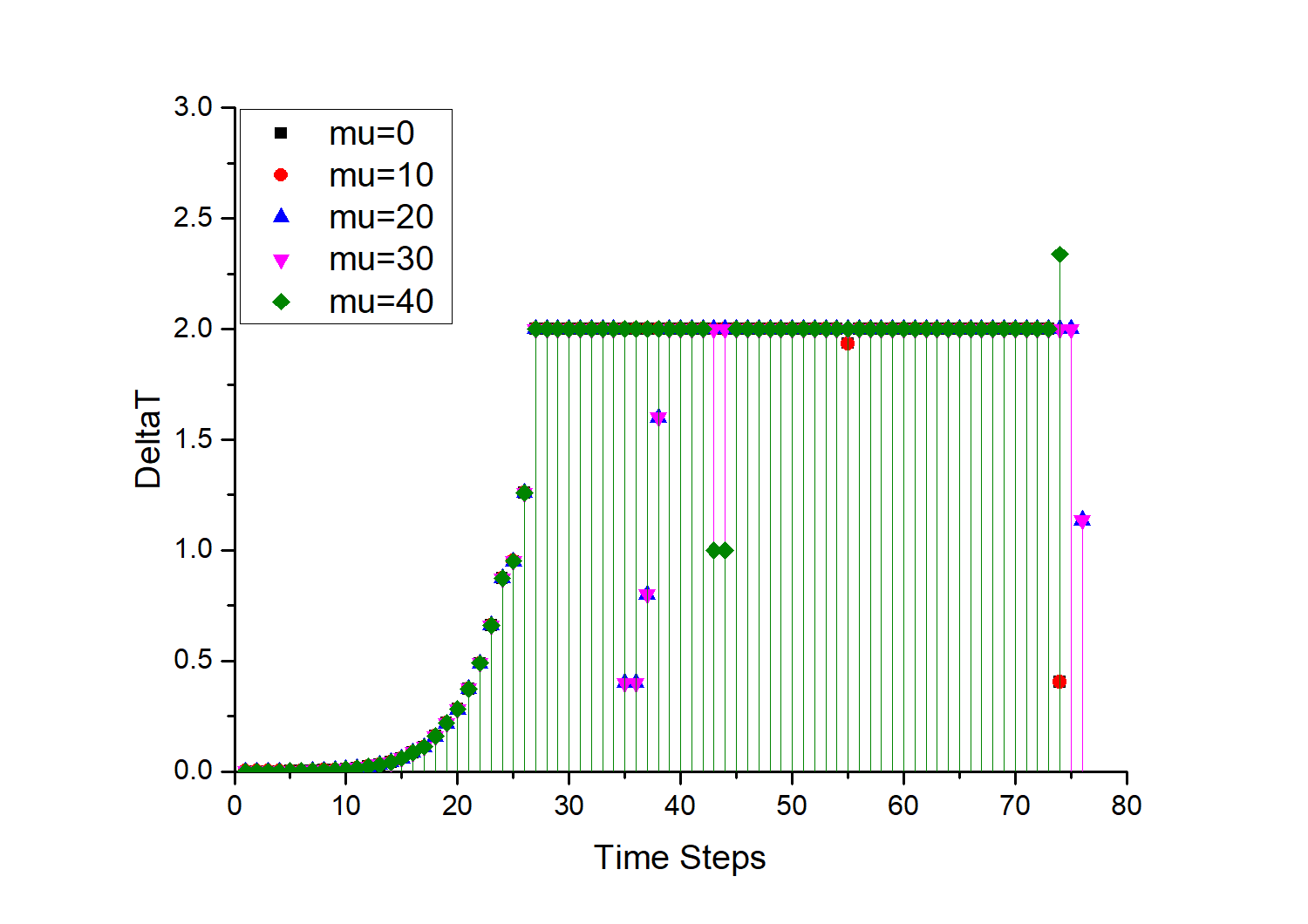}
		\end{minipage}
	}%
	\caption{The timestep size $ \Delta t $ of ASCPR-GMRES-OMP and ASCPR-GMRES-CUDA for the three-phase SPE10 problem.}\label{fig:DeltaT-SPE10-3P}
\end{figure}

\subsubsection{Performance comparisons with commercial simulator}
Also, we compare the experimental results of the OpenMP and GPU versions for commercial and our simulators, respectively. To begin with, \reffigs{fig:SPE10-3P-FOPR-vs-tNav} and \ref{fig:SPE10-3P-FPR-vs-tNav} display the field oil production rate, field gas production rate, and average pressure graphs. Through quantitative comparison, it can be found that our simulation results are very consistent with the results of commercial simulator.

\begin{figure}[H]
	\subfigure[Field oil production rate]{\label{fig:SPE10-3P-FOPR-vs-tNav}
		\begin{minipage}[t]{0.33\linewidth}
			\centering
			\includegraphics[width=5.5cm]{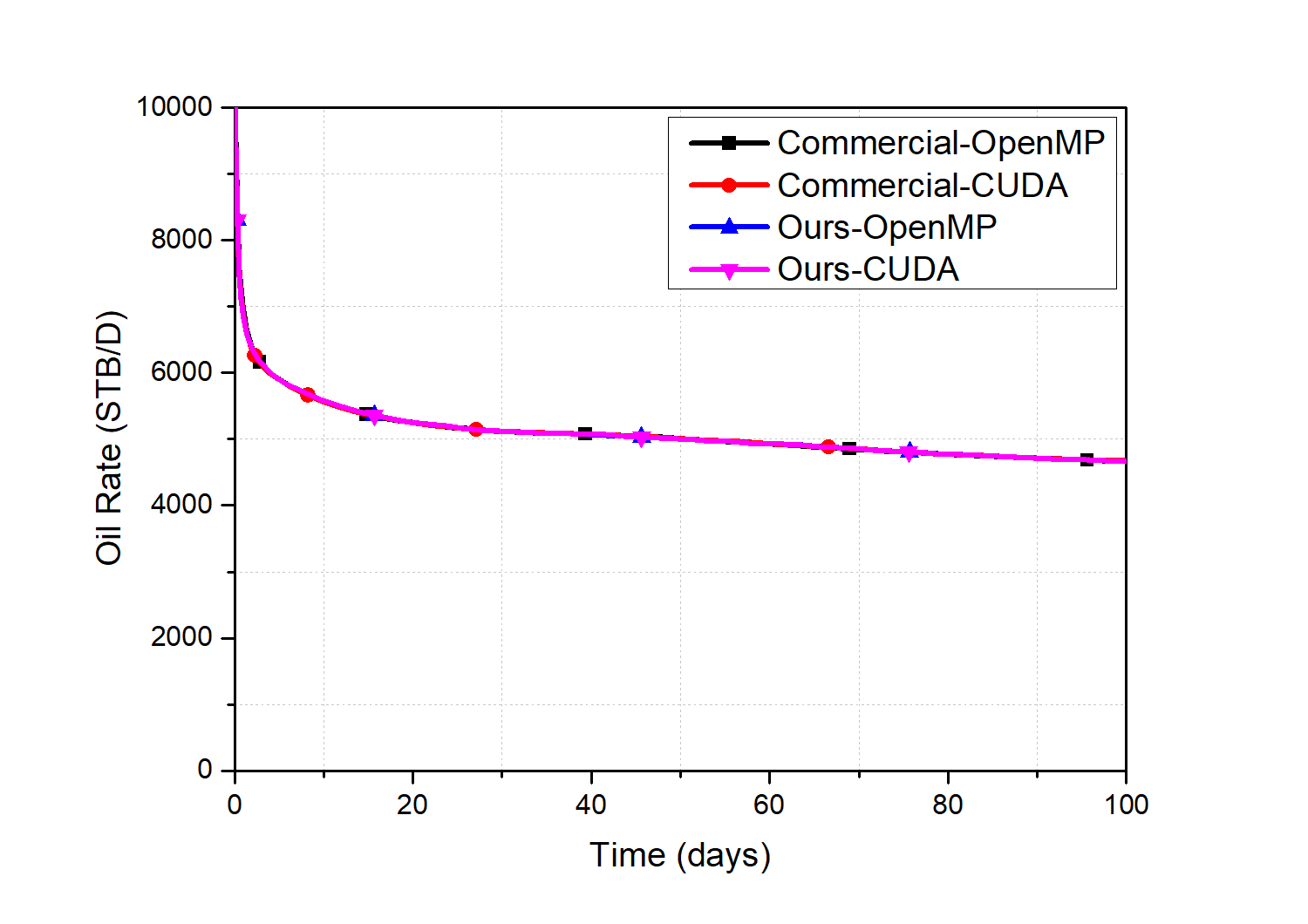}		
		\end{minipage}
	}%	
	\subfigure[Field gas production rate]{\label{fig:SPE10-3P-FGPR-vs-tNav}
		\begin{minipage}[t]{0.33\linewidth}
			\centering
			\includegraphics[width=5.5cm]{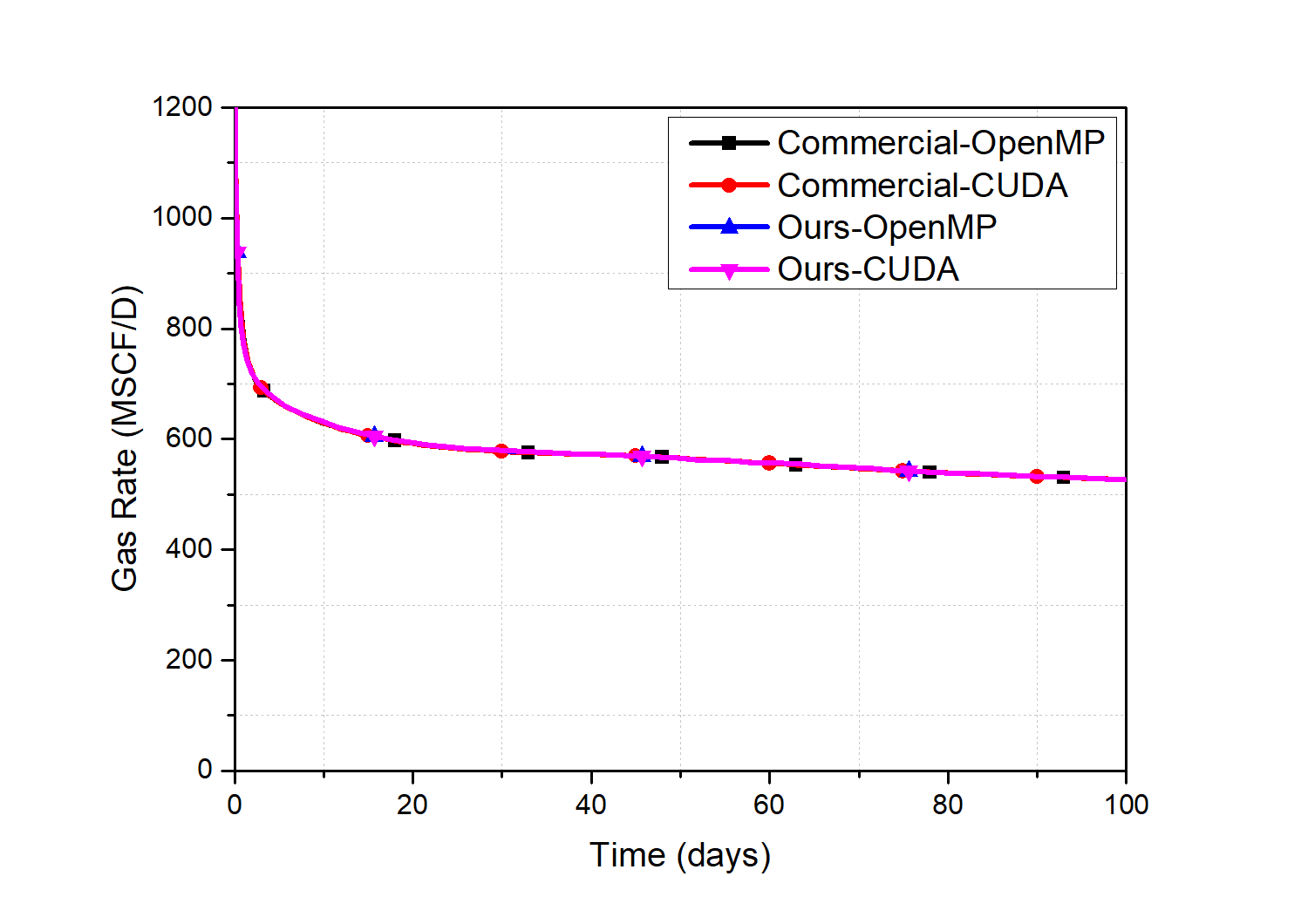}		
		\end{minipage}
	}%
	\subfigure[Average pressure]{\label{fig:SPE10-3P-FPR-vs-tNav}
		\begin{minipage}[t]{0.33\linewidth}
			\centering
			\includegraphics[width=5.5cm]{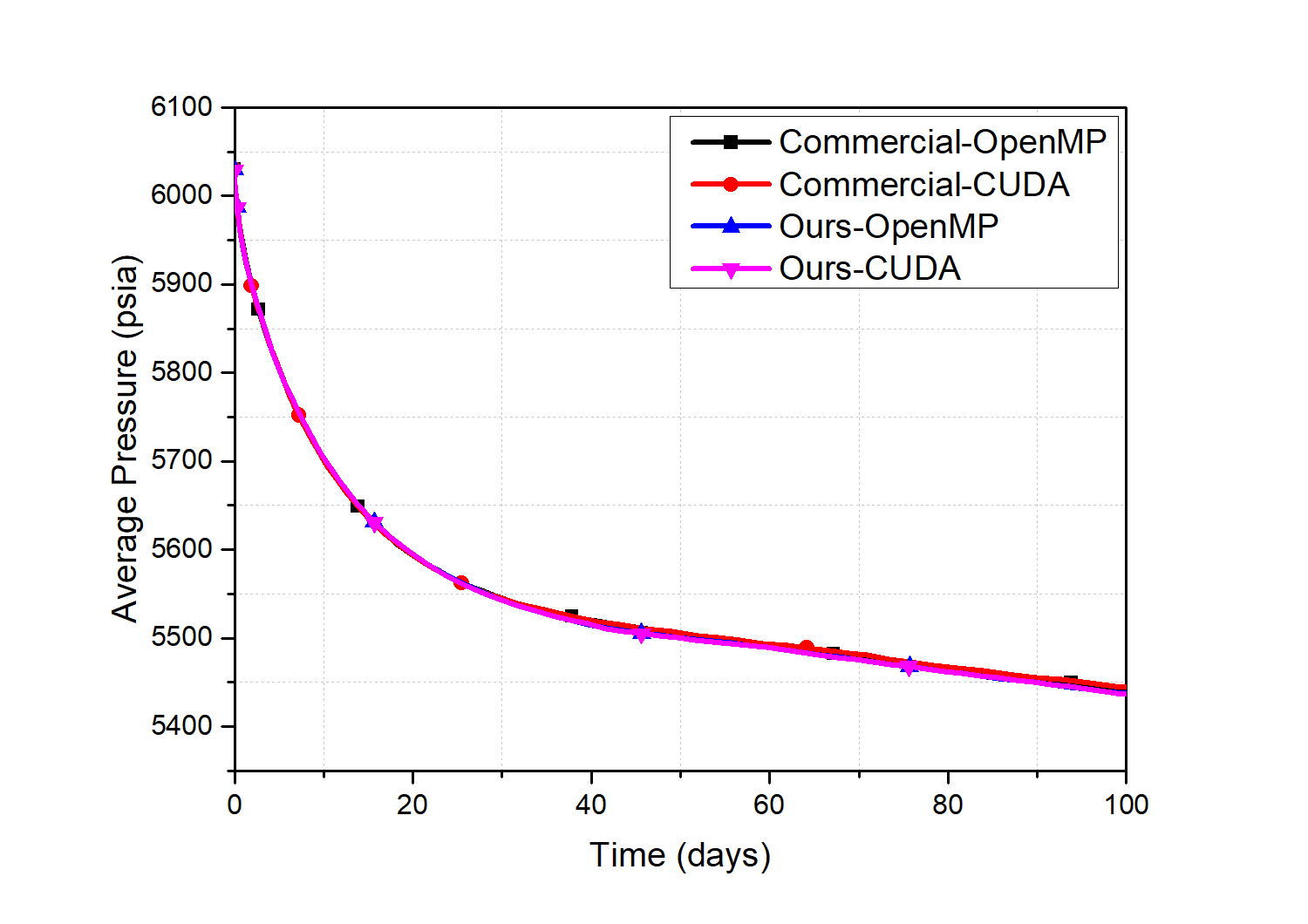}	
		\end{minipage}
	}
	\caption{The field oil production rate, field gas production rate, and average pressure comparison charts of commercial and our simulators for the three-phase SPE10 problem. OpenMP version: NT = 16, Ours: $ \mu=0 $.}
	\label{fig:SPE10-3P-FOPR-FPR-vs-tNav}
\end{figure}

{\small{
		\begin{table}[htpb]
				\centering
				\setlength{\tabcolsep}{12pt}
				\renewcommand\arraystretch{1.2}
				\caption{NumTSteps, NumNSteps, Iter, AvgIter, Time(h), and Speedup comparisons of the OpenMP version of commercial and our simulators for the three-phase SPE10 problem.}
				\label{tab:tNavigator-vs-HiSim-OMP-SPE10-3P}		
				\begin{tabular}{c|c|cccccc}
					\toprule
					Simulators & $\mu$ & NT & 1& 2 & 4 & 8 & 16\\
					\hline
					\multirow{6}{*}{Commercial} &\multirow{6}{*}{---}	
					&\emph{NumTSteps} &114	  &113		&115		&115		&114 		\\
					&&\emph{NumNSteps} &233  	&229		&232		&231		&230 		\\
					&&\emph{Iter}		  &60876	&64693	&67787	&70595	&73864 	\\
					&&\emph{AvgIter}    &261.3 	&282.5 	&292.2 	&305.6 	&321.1	\\	
					&&\emph{Time}     &7.626 	&4.031 	&2.125 	&1.108 	&\textbf{0.590} 	\\
					&&\emph{Speedup}  &1.00		&1.89 	&3.59 	&6.88 	&12.92 	\\
					\hline
					\multirow{30}{*}{Ours} &
					\multirow{6}{*}{0}	
					&\emph{NumTSteps} &73	&73	&73	&73	&73 \\
					&&\emph{NumNSteps} &178 	  &178 	  &178    &178 	  &178  \\
					&&\emph{Iter}		  &3236 	&3236 	&3237 	&3237 	&3237 \\
					&&\emph{AvgIter}    &18.2 	&18.2 	&18.2 	&18.2 	&18.2 \\
					&&\emph{Time}     &1.144 	&0.715 	&0.482 	&0.371 	&0.345\\
					&&\emph{Speedup}  &1.00 	&1.60 	&2.37 	&3.08 	&3.32 \\
					\cline{2-8}
					&\multirow{6}{*}{10}	
					&\emph{NumTSteps} &73	&73	&73	&73	&73 \\
					&&\emph{NumNSteps} &178 	  &178 	  &178    &178 	  &178  \\
					&&\emph{Iter}		  &3253 	&3253 	&3253 	&3253 	&3253 \\
					&&\emph{AvgIter}    &18.3 	&18.3 	&18.3 	&18.3 	&18.3 \\
					&&\emph{Time}     &1.115 	&0.692 	&0.475 	&0.363 	&0.337\\ 
					&&\emph{Speedup}  &1.00 	&1.61 	&2.35 	&3.07 	&3.31 \\
					\cline{2-8}
					&\multirow{6}{*}{20}	
					&\emph{NumTSteps} &73	&73	&73	&73	&73 \\
					&&\emph{NumNSteps} &178 	  &178 	  &178    &178 	  &178  \\
					&&\emph{Iter}		  &3464 	&3463 	&3460 	&3461 	&3460 \\
					&&\emph{AvgIter}   	&19.5 	&19.5 	&19.4 	&19.4 	&19.4 \\
					&&\emph{Time}    	&1.140	&0.699	&0.467	&0.354	&0.319 \\ 
					&&\emph{Speedup} 	&1.00 	&1.63 	&2.44 	&3.22 	&3.57  \\
					\cline{2-8}
					&\multirow{6}{*}{\textbf{30}}	
					&\emph{NumTSteps} &74	&74	&74	&73	&73 \\
					&&\emph{NumNSteps} &178	&178	&178	&178	&178 \\
					&&\emph{Iter}		  &3891 	&3891 	&3890 	&3889 	&3996  \\
					&&\emph{AvgIter}   	&21.9 	&21.9 	&21.9 	&21.8 	&22.5  \\
					&&\emph{Time}    	&1.288 	&0.774 	&0.503 	&0.358 	&\textbf{0.317} \\
					&&\emph{Speedup} 	&1.00 	&1.66 	&2.56 	&3.60 	&4.06  \\
					\cline{2-8}
					&\multirow{6}{*}{40}	
					&\emph{NumTSteps} &74		&74		&74		&74		&74     \\
					&&\emph{NumNSteps} &178		&178	&178	&178	&178    \\
					&&\emph{Iter}		  &4111 	&4111 	&4118 	&4125 	&4105   \\
					&&\emph{AvgIter}   	&23.1 	&23.1 	&23.1 	&23.2 	&23.1   \\
					&&\emph{Time}    	&1.245 	&0.749 	&0.488 	&0.360 	&0.324  \\
					&&\emph{Speedup} 	&1.00 	&1.66 	&2.55 	&3.46 	&3.84   \\
					\bottomrule
				\end{tabular}
		\end{table}
}}

Next, \reftabs{tab:tNavigator-vs-HiSim-OMP-SPE10-3P} and \ref{tab:tNavigator-vs-HiSim-GPU-SPE10-3P} present the number of time steps (\emph{NumTSteps}), the number of Newton iterations (\emph{NumNSteps}), the number of linear iterations (\emph{Iter}), the average number of linear iterations per Newton iteration (\emph{AvgIter}), the total simulation time (\emph{Time}), and the parallel speedup (\emph{Speedup}) for the OpenMP and GPU versions of commercial and our simulators, respectively. 

It can be seen from~\reftab{tab:tNavigator-vs-HiSim-OMP-SPE10-3P} that there is more the number of time steps, Newton iterations, and linear iterations in commercial simulator. When the number of threads increases from 1 to 16, linear iterations increase dramatically (or the average number of linear iterations per Newton iteration increases), and the simulation time is reduced from 7.626 (h) to 0.590 (h). At this point, the maximum speedup reaches 12.92 and the average number of linear iterations per Newton iteration exceeds 260. In our simulator, the number of time steps, Newton iterations, and linear iterations are less. Note that the number of linear iterations is basically stable as the number of threads increases from 1 to 16. When $\mu=30$ and the number of threads varied from 1 to 16, the simulation time was reduced from 1.288 (h) to 0.317 (h). At this point, the maximum speedup reaches 4.06 and the average number of linear iterations per Newton iteration is about 22. This indicates that the proposed method can speed up the simulation time by 1.86 times compared to commercial simulator.

\begin{table}[h]
	\centering
	\setlength{\tabcolsep}{8pt} 
	\renewcommand\arraystretch{1.5} 
	\caption{NumTSteps, NumNSteps, Iter, AvgIter, Time(h), and Speedup (compared with the commercial simulator) comparisons of the GPU version of commercial and our simulators for the three-phase SPE10 problem.}\label{tab:tNavigator-vs-HiSim-GPU-SPE10-3P}	
	\begin{tabular}{c|c|c|c|c|c|c|c}
		\hline
		Simulators	&$\mu$ &\emph{NumTSteps}&\emph{NumNSteps}&\emph{Iter}		&\emph{AvgIter}	&\emph{Time} &\emph{Speedup} \\
		\hline
		Commercial	&---	&117& 233& 71415	 &306.5 &1.004 	&---      \\
		\hline
		\multirow{5}{*}{Ours}
		&0	&74 &178 &3270 &18.4 &0.339 &2.96  \\
		&10	&74 &178 &3302 &18.6 &0.317 &3.17  \\
		&20	&76 &178 &3424 &19.2 &0.282 &3.56  \\
		&\textbf{30}	&76 &178 &4065 &22.8 &\textbf{0.260} &\textbf{3.86}  \\ 
		&40	&74 &178 &4239 &23.8 &0.263 &3.81  \\
		\hline
	\end{tabular}
\end{table}

According to~\reftab{tab:tNavigator-vs-HiSim-GPU-SPE10-3P}, we can summarize similar conclusions. Compared with the commercial simulator, our obtained number of time steps, Newton iterations, and linear iterations is smaller, and the simulation time is shorter. Especially when $\mu=30$, the speedup of our simulator can achieve 3.86 compared to commercial simulator.

\section{Conclusions}\label{sec:7}
In this paper, we investigated an efficient parallel CPR preconditioner for the linear algebraic systems arising from the fully implicit discretization of the black oil model. First, for two difficulties of the preconditioner, the computation cost is large, and the parallel speedup is low in SETUP. We proposed an adaptive SETUP CPR preconditioner to improve the efficiency and parallel performance of the preconditioner. Furthermore, we proposed an efficient parallel GS algorithm based on the coefficient matrix of strong connections. The algorithm can be adapted to unstructured grids, yielded the same convergence behavior as the corresponding single-thread algorithm, and obtained a good parallel speedup. This paper took the CPR preconditioner as an example to illustrate our proposed methods, which can easily be extended to multi-stage preconditioners. In the future, the parallel performance of the adaptive SETUP multi-stage preconditioners needs to be further improved. This paper only considered OpenMP and CUDA implementations for the proposed parallel GS algorithm, so further research will be conducted for MPI parallelism.

\section*{Acknowledgments}
\addcontentsline{toc}{section}{Acknowledgments}
This work was supported by the Postgraduate Scientific Research Innovation Project of Hunan Province (No.~CX20210607) and Postgraduate Scientific Research Innovation Project of Xiangtan University (No.~XDCX2021B110). Feng was partially supported by the Excellent Youth Foundation of SINOPEC (No.~P20009). Zhang was partially supported by the National Science Foundation of China (No.~11971472). Shu was partially supported by the National Science Foundation of China (No.~11971414). We would like to thank Prof. Ruizhong Jiang (China University of Petroleum) for providing numerical results by tNavigator.

\section*{Code availability section}

\begin{itemize}
	\item Name of code: faspsolver, faspcpr
	\item Program language : the code is written in C
	\item Repository: \url{https://github.com/FaspDevTeam/faspsolver}
	\item Repository: \url{https://github.com/zhaoli0321/faspcpr}
	\item Prerequisites: available at the repository
	\item Description: available at the repository
	%	\item Year first available : 2022
\end{itemize}

\bibliographystyle{unsrt}
\bibliography{bibliography}

\begin{thebibliography}{10}

\bibitem{Aziz1979PetroleumRS}
K.~Aziz and A.~Settari.
\newblock {\em Petroleum Reservoir Simulation}.
\newblock Applied Science Publishers, London, 687pp, 1979.

\bibitem{ChenHuanMa}
Z.X. Chen, G.R. Huan, and Y.L. Ma.
\newblock {\em Computational Methods for Multiphase Flows in Porous Media}.
\newblock 2006.

\bibitem{GOUDARZI201696}
A.~Goudarzi, M.~Delshad, and K.~Sepehrnoori.
\newblock A chemical {EOR} benchmark study of different reservoir simulators.
\newblock {\em Computers \& Geosciences}, 94:96--109, 2016.

\bibitem{Peaceman1977}
D.W. Peaceman.
\newblock {\em Fundamentals of Numericl Reservoir Simulation}.
\newblock Developments in Petroleum Science, United States, 190pp, 1977.

\bibitem{VALIOLLAHI2012}
H.~Valiollahi, Z.~Ziabakhsh, and P.L.J. Zitha.
\newblock Mathematical modeling of chemical oil-soluble transport for water
  control in porous media.
\newblock {\em Computers \& Geosciences}, 45:240--249, 2012.

\bibitem{1959One}
J.W. Sheldon, B.~Zondek, and W.T. Cardwell.
\newblock One-dimensional, incompressible, noncapillary, two-phase fluid flow
  in a porous medium.
\newblock {\em Transactions of the AIME}, 216(1):290--296, 1959.

\bibitem{DouglasFIMI}
J.~Douglas, Jr., D.W. Peaceman, and H.H. Rachford, Jr.
\newblock A method for calculating multi-dimensional immiscible displacement.
\newblock {\em Transactions of the AIME}, 216:297--308, 1959.

\bibitem{StoneIMPES}
H.L. Stone and A.O. Garder, Jr.
\newblock Analysis of gas-cap or dissolved-gas drive reservoirs.
\newblock {\em Society of Petroleum Engineers Journal}, 1:92--104, 1961.

\bibitem{AIM}
D.A. Collins, L.X. Nghiem, Y.K. Li, and J.E. Grabonstotter.
\newblock An efficient approach to adaptive-implicit compositional simulation
  with an equation of state.
\newblock {\em SPE Reservoir Engineering}, 7(02):259--264, 1992.

\bibitem{2006Direct}
T.A. Davis.
\newblock {\em Direct Methods for Sparse Linear Systems}.
\newblock 2006.

\bibitem{2003Iterative}
Y.~Saad.
\newblock {\em Iterative Methods for Sparse Linear Systems}.
\newblock Society for Industrial and Applied Mathematics, second edition, 2003.

\bibitem{osti10249810}
J.T. Camargo, J.A. White, N.~Castelletto, and R.I. Borja.
\newblock Preconditioners for multiphase poromechanics with strong capillarity.
\newblock {\em International Journal for Numerical and Analytical Methods in
  Geomechanics}, 45(9):1141--1168, 2021.

\bibitem{Xu1992MSC}
J.C. Xu.
\newblock Iterative methods by space decomposition and subspace correction.
\newblock {\em SIAM Review}, 34(4):581--613, 1992.

\bibitem{10.2118/12262-MS}
J.A. Meyerink.
\newblock Iterative methods for the solution of linear equations based on
  incomplete block factorization of the matrix.
\newblock In {\em SPE Reservoir Simulation Symposium}, volume SPE-12262.
  OnePetro, 1983.

\bibitem{1984Algebraic}
A.~Brandt, S.F. Mccormick, and J.W. Ruge.
\newblock Algebraic multigrid ({AMG}) for sparse matrix equations.
\newblock In {\em Sparsity and its Applications}, pages 257--284. Cambridge
  University Press, 1984.

\bibitem{falgout2006an}
R.D. Falgout.
\newblock An introduction to algebraic multigrid computing.
\newblock {\em Computing in Science and Engineering}, 8(6):24--33, 2006.

\bibitem{10.2118/96809-MS}
H.~Cao, H.A. Tchelepi, J.R. Wallis, and H.E. Yardumian.
\newblock Parallel scalable unstructured {CPR-Type} linear solver for reservoir
  simulation.
\newblock In {\em SPE Annual Technical Conference and Exhibition}, volume
  SPE-96809. OnePetro, 2005.

\bibitem{2017Numerical}
Z.~Li, S.H. Wu, C.S. Zhang, J.C. Xu, C.S. Feng, and X.Z. Hu.
\newblock Numerical studies of a class of linear solvers for fine-scale
  petroleum reservoir simulation.
\newblock {\em Computing \& Visualization in Science}, 18(2-3):93--102, 2017.

\bibitem{IGE1983}
J.R. Wallis.
\newblock Incomplete gaussian elimination as a preconditioning for generalized
  conjugate gradient acceleration.
\newblock In {\em SPE Reservoir Simulation Symposium}. OnePetro, 1983.

\bibitem{10.2118/13536-MS}
J.R. Wallis, R.P. Kendall, and T.E. Little.
\newblock Constrained residual acceleration of conjugate residual methods.
\newblock In {\em SPE Reservoir Simulation Symposium}. OnePetro, 1985.

\bibitem{10.2118/118722-MS}
T.M. Al-Shaalan, H.M. Klie, A.H. Dogru, and M.F. Wheeler.
\newblock Studies of robust two stage preconditioners for the solution of fully
  implicit multiphase flow problems.
\newblock In {\em SPE Reservoir Simulation Symposium}. OnePetro, 2009.

\bibitem{XiaoZheHU2013}
X.Z. Hu, J.C. Xu, and C.S. Zhang.
\newblock Application of auxiliary space preconditioning in field-scale
  reservoir simulation.
\newblock {\em Science China Mathematics}, 56(12):2737--2751, 2013.

\bibitem{10.2118/105832-MS}
K.~St\"{u}ben, T.~Clees, H.~Klie, B.~Lu, and M.F. Wheeler.
\newblock Algebraic multigrid methods ({AMG}) for the efficient solution of
  fully implicit formulations in reservoir simulation.
\newblock In {\em SPE Reservoir Simulation Symposium}. OnePetro, 2007.

\bibitem{BUCKER2008}
H.M. B{\"u}cker, A.I. Kauerauf, and A.~Rasch.
\newblock A smooth transition from serial to parallel processing in the
  industrial petroleum system modeling package petromod.
\newblock {\em Computers \& Geosciences}, 34(11):1473--1479, 2008.

\bibitem{Dogru2009}
A.H. Dogru, L.S.K. Fung, U.~Middya, T.~Al-Shaalan, and J.A. Pita.
\newblock A next-generation parallel reservoir simulator for giant reservoirs.
\newblock In {\em SPE Reservoir Simulation Symposium}, volume SPE-119272.
  OnePetro, 2009.

\bibitem{2014A2}
C.S. Feng, S.~Shu, J.C. Xu, and C.S. Zhang.
\newblock A multi-stage preconditioner for the black oil model and its {OpenMP}
  implementation.
\newblock {\em Lecture Notes in Computational Science and Engineering},
  98:141--153, 2014.

\bibitem{MESBAH2019574}
M.~Mesbah, A.~Vatani, M.~Siavashi, and M.H. Doranehgard.
\newblock Parallel processing of numerical simulation of two-phase flow in
  fractured reservoirs considering the effect of natural flow barriers using
  the streamline simulation method.
\newblock {\em International Journal of Heat and Mass Transfer}, 131:574--583,
  2019.

\bibitem{2010High}
H.~Sudan, H.~Klie, R.~Li, and Y.~Saad.
\newblock High performance manycore solvers for reservoir simulation.
\newblock {\em European Conference on the Mathematics of Oil Recovery}, 2010.

\bibitem{WEI2015}
X.H. Wei, W.S. Li, H.L. Tian, H.L. Li, H.X. Xu, and T.F. Xu.
\newblock {THC-MP}: High performance numerical simulation of reactive transport
  and multiphase flow in porous media.
\newblock {\em Computers \& Geosciences}, 80:26--37, 2015.

\bibitem{WILKINS2021}
A.~Wilkins, C.P. Green, and J.~Ennis-King.
\newblock An open-source multiphysics simulation code for coupled problems in
  porous media.
\newblock {\em Computers \& Geosciences}, 154:104820, 2021.

\bibitem{2014A1}
S.H. Wu, J.C. Xu, C.S. Feng, C.S. Zhang, Q.Y. Li, S.~Shu, B.H. Wang, X.B. Li,
  and H.~Li.
\newblock A multilevel preconditioner and its shared memory implementation for
  a new generation reservoir simulator.
\newblock {\em Petroleum Science}, 11:540--549, 2014.

\bibitem{ShuhongWu2016}
S.H. Wu, B.H. Wang, Q.Y. Li, J.C. Xu, C.S. Zhang, and C.S. Feng.
\newblock Cost-effective parallel reservoir simulation on shared memory.
\newblock In {\em SPE Asia Pacific Oil \& Gas Conference and Exhibition},
  volume SPE-182367-MS. OnePetro, 2016.

\bibitem{YangBo2016}
B.~Yang, H.~Liu, and Z.X. Chen.
\newblock Accelerating linear solvers for reservoir simulation on {GPU}
  workstations.
\newblock {\em Society for Computer Simulation International}, 1:1--8, 2016.

\bibitem{YANG20192}
H.J. Yang, S.Y. Sun, Y.T. Li, and C.~Yang.
\newblock Parallel reservoir simulators for fully implicit complementarity
  formulation of multicomponent compressible flows.
\newblock {\em Computer Physics Communications}, 244:2--12, 2019.

\bibitem{Trangenstein1989}
J.A. Trangenstein and J.B. Bell.
\newblock Mathematical structure of the black-oil model for petroleum reservoir
  simulation.
\newblock {\em SIAM Journal on Applied Mathematics}, 49(3):749--783, 1989.

\bibitem{feng2014numerical}
C.S. Feng, S.~Shu, J.C. Xu, and C.S. Zhang.
\newblock Numerical study of geometric multigrid methods on {CPU-GPU}
  heterogeneous computers.
\newblock {\em Advances in Applied Mathematics and Mechanics}, 6(1):1--23,
  2014.

\bibitem{1993Multi}
V.E. Bulgakov.
\newblock Multi-level iterative technique and aggregation concept with
  semi-analytical preconditioning for solving boundary-value problems.
\newblock {\em Communications in Numerical Methods in Engineering},
  9(8):649--657, 1993.

\bibitem{NapovPariwise}
A.~Napov and Y.~Notay.
\newblock An algebraic multigrid method with guaranteed convergence rate.
\newblock {\em SIAM Journal on Scientific Computing}, 34(2):A1079--A1109, 2012.

\bibitem{HuXiaozhe2013}
X.Z. Hu, P.S. Vassilevski, and J.C. Xu.
\newblock Comparative convergence analysis of nonlinear {AMLI-Cycle} multigrid.
\newblock {\em SIAM Journal on Numerical Analysis}, 51(2):1349--1369, 2013.

\bibitem{2001TenthSPE}
M.A. Christie and M.J. Blunt.
\newblock Tenth {SPE} comparative solution project: A comparison of upscaling
  techniques.
\newblock {\em SPE Reservoir Evaluation \& Engineering}, 4(04):308--317, 2001.

\bibitem{tNavigatorManual2020}
{Rock Flow Dynamics}.
\newblock Reservoir simulator {tNavigator} user manual.
\newblock 2020.

\end{thebibliography}
\addcontentsline{toc}{section}{References}

\end{document}